\documentclass[reqno]{amsart}
\usepackage{amssymb,eucal}

\usepackage[usenames]{color}

\input xy
\xyoption{all}

\xymatrixrowsep{0.2cm}
\xymatrixcolsep{0.2cm}

\usepackage{graphics}

\usepackage{mathrsfs}
\usepackage{relsize}
\usepackage{cases}

\theoremstyle{plain}

\newtheorem{lemma}{Lemma}[section]
\newtheorem{theorem}[lemma]{Theorem}

\newtheorem{corollary}[lemma]{Corollary}
\newtheorem{claim}{Claim}

\newtheorem*{stat}{\name}
\newcommand{\name}{testing}

\newcommand{\N}{\mathbb{N}}

\theoremstyle{definition}
\newtheorem{definition}[lemma]{Definition}

\newtheorem*{problem*}{Problem}

\theoremstyle{remark}
\newtheorem{remark}[lemma]{Remark}
\newtheorem{notation}[lemma]{Notation}

\newtheorem*{remark*}{Remark}

\newcommand{\qedc}{{\qed}~{\rm Claim~{\theclaim}.}}
\newcommand{\qedsc}{{\qed}~{\rm Claim.}}

\numberwithin{equation}{section}

\newcommand{\set}[1]{\{#1\}}

\newcommand{\card}[1]{|{#1}|}

\newcommand{\ba}{\boldsymbol{a}}
\newcommand{\bb}{\boldsymbol{b}}

\newcommand{\bp}{\boldsymbol{p}}
\newcommand{\bq}{\boldsymbol{q}}
\newcommand{\br}{\boldsymbol{r}}

\newcommand{\bx}{\boldsymbol{x}}
\newcommand{\by}{\boldsymbol{y}}

\newcommand{\cA}{\mathcal{A}}

\newcommand{\ins}[1]{\setlength{\fboxsep}{1pt}\colorbox{white}{\text{\scriptsize{#1}}}}

\DeclareMathOperator{\id}{id}
\DeclareMathOperator{\Aut}{Aut}
\DeclareMathOperator{\End}{End}

\newcommand{\ignorer}[1]{}

\newcommand{\blank}{e}

\newcommand{\Z}{\mathbb{Z}}


\begin{document}

\date{\today}

\title{An automaton group with undecidable order and Engel problems}

\author[P.~Gillibert]{Pierre Gillibert}

\address[P. Gillibert]{Institut f\"ur Diskrete Mathematik und Geometrie, Technische Universit\"at Wien, Austria}
\email[P. Gillibert]{pgillibert@yahoo.fr}
\urladdr[P. Gillibert]{http://www.gillibert.fr/pierre/}

\begin{abstract}
For every Turing machine, we construct an automaton group that simulates it. Precisely, starting from an initial configuration of the Turing machine, we explicitly construct an element of the group such that the Turing machine stops if, and only if, this element is of finite order.
If the Turing machine is universal, the corresponding automaton group has an undecidable order problem. This solves a problem raised by Grigorchuk.

The above group also has an undecidable Engel problem: there is no algorithm that, given $g, h$ in the group, decides whether there exists an integer~$n$ such that the $n$-iterated commutator $[\dots [[g,h],h],\dots,h]$ is the identity or not. This solves a problem raised by Bartholdi.
\end{abstract}

\maketitle

\section{Introduction}

Our goal in this paper is to prove the following result, which answers questions raised by Grigorchuk, Nekrashevych, and Sushchanski\u{\i} in \cite[Problem 7.2(a)]{GNS} and \cite[Problem 2(b)]{GrS}

\begin{theorem}
Recognizing in an automaton group whether an element is of finite order is an undecidable problem.
\end{theorem}

We also obtain a negative result for the Engel problem for such groups, raised by Bartholdi in \cite{B1,B2}:

\begin{theorem}
Recognizing in an automaton group whether an $n$-iterated commutator $[\cdots[[v, w], w], \ldots, w]$ is trivial for some~$n$ is an undecidable problem.
\end{theorem}

As an application of Theorem~1.1, we obtain a new example of a finitely presented group with decidable word problem and undecidable order problem (McCool's theorem \cite{McCool2}). Indeed, as an automaton group is finitely generated and has decidable word problem, it has a recursive presentation with a finite set of generators. It follows from Higman's representation theorem \cite{Higman} that an automaton group can be embedded in a finitely presented group. Moreover, using Clapham's theorem \cite{Clapham}, we can choose a finitely presented group with a decidable word problem. Hence our automaton group embeds in a finitely presented group with a decidable word problem and an undecidable order problem. 

\subsection*{The principle of our proof}

Our method is to show that an arbitrary cellular automaton can be simulated using a Mealy automaton. As all Turing machines can be simulated using cellular automata, we deduce, for every Turing machine $M$, a Mealy automaton that generates an automaton group~$G(M)$, and we can prove that, for each word $w$ there is a (computable) element~$g$ of~$G(M)$ such that $M$ stops on $w$ if, and only if,~$g$ is of finite order (then a power of $2$) in~$G(M)$. As the halting problem is undecidable (Turing \cite{T}), if $M$ is a universal Turing machine, the order problem in the  associated automaton group $G(M)$ is undecidable (Theorem~1.1). Moreover, there exists a special element $h$ such that, for all possible elements~$g$ constructed above, we have $[\dots,[[g,h],h],\dots,h]=g^{\pm 2^n}$ ($n$~commutators). It follows that the Engel problem for the group~$G(M)$ is also undecidable (Theorem~1.2). 

The principle for constructing the group~$G(M)$ is as follows. First, we note that a Turing machine can be simulated using a cellular automaton. Call its alphabet~$T$, and add a separator symbol $\$$ and a counter on each symbols, that is, put $\Sigma=(T\cup\set{\$})\times\set{0,1}$. For the moment, we skip the description of the state set. The point is that, for every~$w$ in~$T^*$, we can construct an element~$g$ of the automaton group with the following behavior. Set $w_0=w$, and define inductively $w_{k+1}$ as the finite configuration (increasing the length by one) obtained by applying one step of the cellular automaton to $w_k$ (completed with some default symbol). Consider the word $w_0\$ w_1 \$\cdots $ and set the counters to zero, obtaining a word~$\bx$ in $\Sigma^{\omega}$. The construction of~$g$ is such that, when~$g$ is applied once to~$\bx$, it only modifies the counters and keeps them synchronized. If we take any element $\by$ that is not in the~$g$-orbit of~$\bx$, the counters are not synchronized, at least one letter is different, or one specific end-state appears. Then~$g$, applied enough times, ``finds'' the difference and reverses the counter, so we obtain $g^{2^n}(\by)=\by$, for some $n$ only depending on the position of the first discrepancy. The undecidability of the order problem follows.

The key point in our encoding is to ensure that updating once a configuration of the initial cellular automaton corresponds to \emph{squaring} the associated element in the final group. This explains how elements whose order in the group is a high power of $2$ arise, making the order problem of that group undecidable.

\subsection*{The general context}

To give a general perspective, let us review some of the many known decidability and, mainly, undecidability results involving automata groups and semigroups. We recall that a \emph{Mealy automaton} is an automaton with an additional output map. Each state of such an automaton gives rise to an endomorphism of the tree of  all  words in the involved alphabet. By definition, the (semi)group generated by such endomorphisms is an \emph{automaton (semi)group}.

Automaton groups proved to be efficient tools for constructing counterexamples to various conjectures. Ale\v sin in \cite{A} gives an example of an infinite, finitely generated torsion group (\emph{Burnside} group for short). Golod and Shafarevich already constructed a Burnside group in \cite{Go,GoS}, however Ale\v sin's group is much simpler, being a $2$-generator subgroup of an automaton group. Similarly, using a Mealy automaton, Sushchanski\u{\i}~\cite{S} constructed for every prime~$p$ a Burnside $p$-group.

Grigorchuk gave another example of a Burnside 2-group in \cite{Gr1}; in this case, the group is the whole automaton group. This group has many remarkable properties. Grigorchuk proved in~\cite{Gr2} that the group has intermediate growth (neither polynomial, nor exponential), solving a problem by Milnor \cite[Problem~5603]{M}. The group is also amenable but not elementary amenable, answering a question by Day \cite{Day}. Grigorchuk generalized his construction obtaining in \cite{Gr3} and obtained $p$-groups with intermediate growth. We refer to \.Zuk in~\cite{Z} for further examples of interesting automaton groups, see also Brough and Cain \cite{BC1, BC2} for several examples of automaton semigroup constructions.

We proved in~\cite{G1} that the finiteness problem and the order problem for reset automaton semigroups are undecidable. Our proof consists in reducing the problems to the tiling problem by a NW-deterministic tile set, which is proved to be undecidable by Kari in~\cite{Kari1}. The encoding is such that any large product of generators either yields a (quasi) nilpotent element, or encodes a large tiling of the plane. Using similar techniques, we could encode the computation of an arbitrary one-way cellular automaton, see for example Delacourt and Ollinger \cite{DO}.

Concerning the order problem, which we shall tackle below, and closely related problems, several undecidability results are known. Belk and Bleak constructed in \cite{BB} a group generated by an \emph{asynchronous} automaton with an undecidable order problem. However, the construction relies on asynchronicity deeply (typically, at the expense of forgetting the first letter, one can simulates any reversible Turing machine). 

Delacourt and Ollinger proved in \cite{DO} that the periodicity problem for one-dimensional one-way cellular automata starting from a given initial almost constant configuration is undecidable. They also pointed that the uniform periodicity problem is equivalent to the finiteness problem for reset automaton groups, see also Silva and Steinberg \cite[Theorem 4.2]{SilvaSteinberg}. Kari and Ollinger proved in \cite{KO} that the periodicity problem for one-dimensional cellular automata is undecidable. We refer to Kari \cite{Kari2,Kari3,Kari4,Kari5} for more related results for cellular automata.

On the other hand, there are positive results. Bondarenko, Bondarenko, Sidki, and Zapata proved in \cite{BBSZ} that the order problem is solvable in bounded automaton groups.

For reversible Mealy automata, although the order problem was not solved so far, several positive results have been established. It is worth noting that all results mentioned in this paragraph are effective, thus providing decidability criteria. Klimann proved in \cite{K} that an automaton semigroup generated by a 2-state reversible Mealy automaton is either finite or free. Godin, Klimann, and Picantin studied in \cite{GKP} the torsion part of the automaton semigroup associated with an invertible reversible automaton. In particular an invertible reversible Mealy automaton with no bireversible component generates a torsion-free semigroup. Similarly Klimann, Picantin, and Savchuk gave in \cite{KPS2} a method to find elements of infinite order in some automaton groups. Finally, Godin and Klimann proved in \cite{GK} that a connected bireversible Mealy automaton of prime size generates either a finite group or a group with an element of infinite order, extending a result of Klimann, Picantin, and Savchuk in \cite{KPS1}. 

Although the classes of reset and reversible Mealy automata are disjoint, some of their properties are similar. In particular, Olukoya proved in \cite{Olukoya} that an automaton group generated by a reset Mealy automaton is either finite or contains a free semigroup over two generators, and, therefore, it has exponential growth and contains an element of infinite order. However, it is not known whether those conditions are decidable for groups, and they are known to be undecidable for semigroups.

On the other hand, it is well known that the word problem is solvable for automaton groups. The argument relies on the Eilenberg's reduction algorithm of \cite[Chapter XII]{E}. Also see Steinberg \cite{Steinberg}, D'Angeli, Rodaro, and W\"achter \cite{DRW} for more about the complexity of the word problem. For further decidability results, we refer to Akhavi, Klimann, Lombardy, Mairesse, and Picantin \cite{AKLMP}, to Bartholdi \cite{B1}, or to Grigorchuk and \v Suni\'c \cite{GrS} for related questions. Furthermore, \v Suni\'c and Ventura constructed  in \cite{SV} an automaton group with undecidable conjugacy problem, and D'Angeli, Godin, Klimann, Picantin, and Rodaro proved in \cite{DR,DGKPR} that the existence of non-elementary commuting pairs of elements is undecidable.

\subsection*{About this text}

The constructions of this paper and the undecidability of the order problem for automaton groups were exposed in July 2017 at the MealyM Final Event \cite{GMealyM}. After the first version of this paper was posted~\cite{G2}, Laurent Bartholdi and Ivan Mitrofanov announced another proof, extending the undecidability result to contracting groups generated by Mealy automata~\cite{BartholdiMitrofanov} and appealing to an encoding of Minsky machines. In this way, they deduce the undecidability of the word problem in self-similar groups.

On the other hand, in an ongoing work with Dmitri Savchuk, we use an encoding of cellular automaton similar to the one of the current paper to obtain a self-similar group with an undecidable word problem.

\subsection*{Acknowledgement}

I thank Ines Klimann and Matthieu Picantin for helpful comments, suggestions, and corrections.

\section{Basic concepts and notations}\label{S:bcn}

The three main decision problems investigated in this paper are:

\begin{problem*}
Let~$G$ be a group, generated by a finite set $\Sigma$.
\begin{enumerate}
\item Word problem: given a word  $w\in\Sigma^*$, does $w$ represent the identity in~$G$?
\item Order problem: given a word $w\in\Sigma^*$, does $w$ represent an element of finite order in~$G$?
\item Engel problem: given words $v,w\in\Sigma^*$, does there exist $n\ge 1$ such that the $n$-iterated commutator $[\dots [[v,w],w],\dots,w]$, represents the identity in~$G$?
\end{enumerate}
\end{problem*}

Note that the decidability of the three problems does not depend on the choice of the generating set: if a problem is decidable for a generating set $\Sigma$, then it is also decidable for any other finite generating set. As mentioned by McCool in \cite{McCool1}, if the order of an element can be computed, then the word problem is solvable. Conversely, if the order problem and the word problem are solvable, then the order can be computed.

When an alphabet is given, we generally use thin symbols $a,b,p,q,x,y$ for the letters and bold symbols $\ba,\bb,\bp,\bq,\bx,\by$ for the words in the alphabet. We denote by $\Sigma^*$ the set of all words in the alphabet $\Sigma$.

A \emph{Mealy automaton} is a tuple $\cA=(A,\Sigma,\delta,\sigma)$, where $A$ and $\Sigma$ are finite sets and $\delta\colon A\times \Sigma\to A$ and $\sigma\colon A\times \Sigma\to \Sigma$ are maps. The elements of $A$ are called the \emph{states} of~$\cA$, while $\Sigma$ is the \emph{alphabet} of $\cA$. The map $\delta$ is the \emph{transition map}, and the map $\sigma$ is the \emph{output map}. We extend $\delta\colon A^*\times\Sigma^*\to A^*$, and $\sigma\colon A^*\times\Sigma^*\to \Sigma^*$, using \eqref{E:defext1}--\eqref{E:defext4} inductively, where we put $\sigma_{\bp}(\bx)=\sigma(\bp,\bx)$, and $\delta_{\bx}(\bp)=\delta(\bp,\bx)$ for all $\bp\in A^*$ and $\bx\in\Sigma^*$. Note that $\sigma$ defines a right action of $A^*$ on $\Sigma^*$.
\begin{equation}\label{E:defext1}
\sigma_{\bp}(\bx\by)=\sigma_{\bp}(\bx)\sigma_{\delta_{\bx}(\bp)}(\by)\,, \quad\text{for all $\bx\in A^*$, $\bp\in\Sigma^*$, and $\by\in\Sigma^*$.} 
\end{equation}
\begin{equation}\label{E:defext2}
\delta_{\bx}(\bp\bq)=\delta_{\bx}(\bp)\delta_{\sigma_{\bp}(\bx)}(\bq)\,, \quad\text{for all $\bx\in \Sigma^*$, $\bp\in A^*$, and $\bq\in A^*$.} 
\end{equation}
\begin{equation}\label{E:defext3}
\sigma_{\bp\bq}=\sigma_{\bq}\circ\sigma_{\bp}\,, \quad\text{for all $\bp,\bq\in A^*$.} 
\end{equation}
\begin{equation}\label{E:defext4}
\delta_{\bx\by}=\delta_{\by}\circ\delta_{\bx}\,, \quad\text{for all $\bx,\by\in \Sigma^*$.} 
\end{equation}

The \emph{semigroup generated by $\cA$} is the subsemigroup of $\End\Sigma^*$ generated by all $\sigma_p$ with $p\in A$. If all the maps $\sigma_p$ with $p\in A$ are bijective, then the \emph{group generated by $\cA$} is the subgroup of $\Aut\Sigma^*$ generated by all $\sigma_p$ with $p\in A$. An \emph{automaton group} is a group generated by a Mealy automaton.

In order to visualize the image of $\sigma$ and $\delta$, we use the notation of \emph{cross diagrams} defined below. For all $\bp,\bp'\in A^*$ and all $\bx,\bx'\in\Sigma^*$, we say that the following cross diagram holds
\[
 \xymatrix{
	& \bx \ar[dd] &\\
    \bp \ar[rr]  & &  \bp'\\
    & \bx'
  }
\]
if, and only if, $\bx'=\sigma_{\bp}(\bx)$ and $\bp'=\delta_{\bx}(\bp)$. The equalities \eqref{E:defext1}--\eqref{E:defext4} are summarized in the following cross diagrams
\[
 \xymatrix{
	& \bx \ar[dd] \\
    \bp  \ar[rr]  & &  \delta_{\bx}(\bp) \\
    & \sigma_{\bp}(\bx)\ar[dd] & &  \\
	\bq \ar[rr] & &  \delta_{\sigma_{\bp}(\bx)}(\bq) \\
	&  \sigma_{\bp\bq}(\bx) & &
  }
\quad
 \xymatrix{
	& \bx \ar[dd] & & \by \ar[dd] &\\
    \bp  \ar[rr]  & &  \delta_{\by}(\bp) \ar[rr]  & &  \delta_{\bx\by}(\bp)\\
    & \sigma_{\bp}(\bx) & & \sigma_{\delta_{\bx}(\bp)}(\by)
  }
\]

Given $\bp,\bq\in A^*$, we write $\bp\equiv \bq$ if, and only if, $\sigma_{\bp}=\sigma_{\bq}$. We write $\bp\equiv\id$ if, and only if, $\sigma_{\bp}$ is the identity.

Given a group~$G$, and $g,h$ in~$G$, the \emph{commutator of~$g$ and $h$} is $[g,h]=g^{-1}h^{-1}gh$. We say that $(g,h)$ is an \emph{Engel} pair if there exists an integer $n\ge 1$ such that the $n$-iterated commutator $[\dots [[g,h],h],\dots,h]$ is the identity.

Let $m\le n$ be integers. A \emph{one-dimensional cellular automaton with neighborhood $\set{m,m+1,\dots,n-1,n}$} is a tuple $(T,t)$, where $T$ is a finite set and $t\colon T^{m-n +1}\to T$ is a partial map. We say that $(T,t)$ is \emph{complete} if $t$ is defined everywhere. We call the element of $T$ the \emph{states} of the cellular automaton, we call $t$ the \emph{transition map}. A \emph{configuration} of $(T,t)$ is a map $c\colon \Z\to T$. Given a configuration $c$, the \emph{updated} or \emph{next} configuration is $c'\colon \Z\to T$, defined  for all $k\in\Z$ by 
$$c'(k) = t(c(k+m),c(k+m+1),\dots,c(k+n-1),c(k+n-1)).$$

Due to the local nature of a Turing machine computation, for each Turing Machine $M$, we can construct a cellular automaton such that each configuration of $M$ can be encoded in a configuration of the cellular automaton, and such that the processes of updating the configuration for $M$ and for the cellular automaton coincide. We refer to the beginning of Section~\ref{S:indecidabilite} for references.

Unless otherwise specified, a \emph{cellular automaton} is a one-dimensional complete cellular automaton with neighborhood $\set{-1,0,1}$. Consider a cellular automaton $(T,t)$, with a distinguished \emph{default} state $\blank\in T$, and a set of \emph{final} states $F\subseteq T$. A \emph{finite configuration} of $(T,t)$ is a word $\ba=a_1\dots a_n\in T^*$. To update a finite configuration, we first pad the configuration with the symbol $\blank$, update the configuration as an infinite configuration, then restrict the result to a finite configuration. The exact definition is:

\begin{notation}\label{N:applyingcellularautomata}
Let $\ba=a_1\dots a_n\in T^*$ be a finite configuration. Set $a_0=a_{n+1}=a_{n+2}=\blank$. For all $1\le k\le n+1$ set $b_k=t(a_{k-1},a_k,a_{k+1})$. The \emph{updated} or \emph{next} configuration is $\tau(\ba)=b_1\dots b_n b_{n+1}$.
\end{notation}

This construction is easier to understand in terms of Turing machines. Every configuration of a Turing machine only needs a finite tape, but, as the computation goes on, the amount of required memory increases. The right-left asymmetry in the definition corresponds to Turing machine working on a one-way infinite tape.

A cellular automaton \emph{halts} from the finite initial configuration $\ba=a_1\dots a_n$ if there exists $k\in\N$ such that a symbol of $F$ appears in $\tau^k(\ba)$. It is known that there exists a universal cellular automaton for which the halting problem, from a given finite configuration, is undecidable. We give an example of a universal cellular automaton in Section~\ref{S:indecidabilite}.

\section{Simulating Cellular Automata}\label{S:construction}

Starting from a cellular automaton, we construct a Mealy automaton simulating transitions between finite configurations. In this section we only give the construction, a few basic properties, and the main tool needed to construct elements of the group simulating the cellular automaton.

Fix a cellular automaton $(T,t)$, where $t\colon T^3\to T$. Fix $\blank\in T$ and $F\subseteq T$. Take a symbol $\$ $ which is not in $T$. Set $T'=T\cup\set{\$}$.

The state $\blank$ should be understood as the blank or default state (i.e. finite configurations of the cellular automaton are completed with this state). The states of $F$ should be understood as final states.

Let $C,W,P$, and $\neg$ be new symbols. Consider the Mealy automaton which has alphabet $\Sigma=T'\times\set{0,1}$, and state set
\[
A=\set{\neg}\cup(\set{C}\times T \times T')\cup (\set{W}\times T \times T) \cup (\set{P}\times T)
\]

The states in $\set{C}\times T \times T'$ are called \emph{counting} states. A \emph{modifying} state is a state which is either $\neg$ or a counting state. Thus, each counting state is a modifying state and $\neg$ is the only modifying non-counting state, all other states are non-modifying (and thus also non-counting). We put $C_a^b=(C,a,b)\in \set{C}\times T \times T'$, and, similarly, $W_a^b=(W,a,b)\in \set{W}\times T \times T$, and $P_a=(P,a)\in \set{P}\times T$.

Define $\bot\colon\Sigma \to\Sigma$ by
$$x=(a,\varepsilon) \mapsto x^{\bot}=(a,1-\varepsilon).$$
So, we have $(a,0)^\bot=(a,1)$, and $(a,1)^\bot=(a,0)$. In particular, the map~$\bot$ defines a bijection. Consider now the map $\sigma \colon A\times\Sigma \to\Sigma$ defined by
$$(s,x) \mapsto 
\begin{cases}
x^\bot\,, &\text{if $s$ is a modifying state.}\\
x\,, &\text{otherwise.}
\end{cases}$$

All modifying states have the same action on one letter (exchanging 0 and 1). All non-modifying states act like identity on one letter.

Next, define $\delta\colon A\times\Sigma \to A$, for all $(p,x)\in A\times\Sigma$, all $a\in T$, all $b,d\in T'$, and all $\varepsilon\in\set{0,1}$, in the following way, where we agree that  \eqref{E:deltaNegouF} has priority over all other clauses when it applies:
\begin{numcases}{\delta(p,x)=}
\neg                            &\text{if $p=\neg$ or $x\in F\times\set{0,1}$.}\label{E:deltaNegouF}\\
C_\blank^\$			&\text{if $p=C_a^\$$ and $x=(\$,1)$.}\label{E:deltaC1}\\
\neg				&\text{if $p=C_a^b$ and $x=(\$,1)$, with $b\not=\$ $.}\label{E:deltaC2}\\
C_\blank^{t(a,\blank,\blank)}  &\text{if $p=C_a^\$$ and $x=(\$,0)$.}\label{E:deltaC3}\\
C_d^b				&\text{if $p=C_a^b$ and $x=(d,1)$, with $d\not=\$ $.}\label{E:deltaC4}\\
\neg				&\text{if $p=C_a^b$ and $x=(d,0)$, with $b\not=d $.}\label{E:deltaC5}\\
W_a^b				&\text{if $p=C_a^b$ and $x=(b,0)$, with $b\not=\$ $.}\label{E:deltaC6}\\
C_\blank^{t(a,b,\blank)}	&\text{if $p=W_a^b$ and $x=(\$ ,\varepsilon)$.}\label{E:deltaW1}\\
P_{t(a,b,d)}			&\text{if $p=W_a^b$ and $x=(d,\varepsilon)$, with $d\not=\$ $.}\label{E:deltaW2}\\
C_\blank^a			&\text{if $p=P_a$ and $x=(\$,\varepsilon)$.}\label{E:deltaP1}\\
P_a				&\text{if $p=P_a$ and $x=(d,\varepsilon)$, with $d\not=\$ $.}\label{E:deltaP2}
\end{numcases}

We extend $\sigma\colon A^*\times\Sigma^*\to\Sigma^*$ and $\delta\colon A^*\times\Sigma^*\to A^*$ using \eqref{E:defext1}--\eqref{E:defext4}. 
Herafter, and in the rest of the paper, we denote by $\cA=(A,\Sigma,\delta,\sigma)$ the Mealy automaton we are constructing, and by $G$ the associated automaton group

The name \emph{counting state} can be understood when we look at the action on a word. In a ``generic everything is working case'', a counting state simply adds 1 in binary, of course depending on the control character the action will be different. The only modifying non-counting state is~$\neg$, which acts like a negation (keeps exchanging 0 and 1 for the whole word). A non-modifying state does not change the word, until it reads some specific control character.

\begin{remark}\label{R:faitbasesurdelta}
When reading the various cases in the definition of $\delta$, we see that the following statements hold, for all $p\in A$ and all $x\in\Sigma$.
\begin{enumerate}
\item If $p$ is not counting, then $\delta_x(p)=\delta_{x^\bot}(p)$. See \eqref{E:deltaNegouF}, \eqref{E:deltaW1}, \eqref{E:deltaW2}, \eqref{E:deltaP1}, and \eqref{E:deltaP2}.

\item If $p$ is not counting and $x\in T\times\set{0,1}$, then $\delta_x(p)=\delta_{x^\bot}(p)$ is not counting. See \eqref{E:deltaNegouF}, \eqref{E:deltaW2}, and \eqref{E:deltaP2}.

\item If $x=(a,0)$, with $a\not=\$ $, then $\delta_x(p)$ is not counting. See \eqref{E:deltaNegouF}, \eqref{E:deltaC5}, \eqref{E:deltaC6}, \eqref{E:deltaW2}, and \eqref{E:deltaP2}.

\item If $p$ is a modifying state and $x=(a,1)$, then $\delta_x(p)$ is a modifying state. See \eqref{E:deltaNegouF}, \eqref{E:deltaC1}, \eqref{E:deltaC2}, and \eqref{E:deltaC4}.

\item If $x\in (T\setminus F)\times\set{0,1}$ and $p$ is non-modifying, then $\delta_x(p)$ is non-modifying. See \eqref{E:deltaW1}, \eqref{E:deltaW2}, \eqref{E:deltaP1}, and \eqref{E:deltaP2}.

\item If $p$ is not in $\set{C}\times T\times\set{\$}$, then $\delta_x(p)$ is not in $\set{C}\times T\times\set{\$}$. See \eqref{E:deltaNegouF}, \eqref{E:deltaC2}, \eqref{E:deltaC4}, \eqref{E:deltaC5}, \eqref{E:deltaC6}, \eqref{E:deltaW1}, \eqref{E:deltaW2}, \eqref{E:deltaP1}, and \eqref{E:deltaP2}.

\item If $x=(\$,\varepsilon)$, $p$ is a modifying state, and $p$ is not in $\set{C}\times T\times\set{\$}$, then $\delta_x(p)=\neg$. See \eqref{E:deltaNegouF}, \eqref{E:deltaC2}, and \eqref{E:deltaC5}.
\end{enumerate}
\end{remark}

As $\bot$ is an involution, the action of a word in $A^*$ on a single letter in $\Sigma$ is straightforward to describe.

\begin{lemma}\label{L:modifyingword}
For $\bp\in A^*$ and $x\in\Sigma$, the equality $\sigma_{\bp}(x)=x^\bot$ holds if, and only if, the number of modifying states in $\bp$ is odd.
\end{lemma}

The following lemma illustrates how $\neg$ reverses the simulation.

\begin{lemma}\label{L:inversestate}
For every $p\in A$, we have $p\neg p\equiv \neg$.
\end{lemma}

\begin{proof}
Let $p\in A$ and $x\in\Sigma$. Note that, in $p\neg p$, the number of modifying states is odd, hence Lemma~\ref{L:modifyingword} implies $\sigma_{p\neg p}(x)=x^\bot=\sigma_{\neg}(x)$.

Assume there exists an integer $n$ such that $\sigma_{p\neg p}(\bx)=\sigma_{\neg}(\bx)$ for all $\bx\in\Sigma^n$ and all $p\in A$.

Let $p\in A$. Let $x\in\Sigma$ and $\by\in\Sigma^n$. Assume that $p$ is not a modifying state. Therefore, by Remark~\ref{R:faitbasesurdelta}(1), the following equalities hold
\[
\delta_{x}(p\neg p) = \delta_x(p)\delta_{\sigma_{p}(x)}(\neg p) = \delta_x(p)\delta_{x}(\neg p) = \delta_x(p)\neg\delta_{x^\bot}(p) = \delta_x(p)\neg\delta_{x}(p)\,.
\]
On the other hand, if $p$ is a modifying state, we find
\[
\delta_{x}(p\neg p) = \delta_x(p)\delta_{\sigma_{p}(x)}(\neg p) = \delta_x(p)\delta_{x^{\bot}}(\neg p) = \delta_x(p)\neg\delta_{\sigma_{\neg}(x^\bot)}(p) = \delta_x(p)\neg\delta_{x}(p)\,.
\]
Therefore, in both cases, we have $\delta_{x}(p\neg p) = \delta_x(p)\neg\delta_{x}(p)$. Thus the following equalities hold
\[
\sigma_{p\neg p}(x\by) = \sigma_{p\neg p}(x)\sigma_{\delta_x(p\neg p)}(\by) = \sigma_{\neg}(x)\sigma_{\delta_x(p)\neg \delta_x(p)}(\by) = \sigma_{\neg}(x)\sigma_{\neg}(\by)=\sigma_{\neg}(x\by)\,.
\]
The result then follows using an induction.
\end{proof}

As an immediate consequence of Lemma \ref{L:inversestate}, we obtain the following description of inverses.

\begin{corollary}\label{C:inversemotsym}
Let $\bp\in A^*$ be a palindrome. Then the inverse of $\sigma_{\bp}$ is $\sigma_{\neg \bp \neg}$.
\end{corollary}

Furthermore, we obtain a connection between commutators and powers.

\begin{corollary}\label{C:conjuguepuissance}
Let $\bp,\bq$ be palindromes in $A^*$. Let $n\ge 1$ be an integer. Then $[\dots [[\sigma_{\bp\bq},\sigma_{\neg\bq}],\sigma_{\neg\bq}],\dots,\sigma_{\neg\bq}]\equiv (\sigma_{\bp\bq})^{(-2)^n}$, where there are $n$ commutators.
\end{corollary}

\begin{proof}
Let $k$ be an integer. The following equalities hold
\begin{align*}
[\sigma_{\bp\bq}^k,\sigma_{\neg\bq}]
&= \sigma_{\bp\bq}^{-k} \circ \sigma_{\neg\bq}^{-1} \circ \sigma_{\bp\bq}^{k} \circ \sigma_{\neg\bq}\\
&= \sigma_{\bp\bq}^{-k} \circ ( \sigma_{\bq} \circ \sigma_{\neg})^{-1} \circ (\sigma_{\bq}\circ \sigma_{\bp})^k   \circ \sigma_{\bq} \circ\sigma_{\neg}\,, &&\text{by \eqref{E:defext3}.}\\
&= \sigma_{\bp\bq}^{-k}\circ  \sigma_{\neg} \circ  \sigma_{\bq}^{-1}   \circ (\sigma_{\bq}\circ \sigma_{\bp})^k   \circ \sigma_{\bq} \circ\sigma_{\neg} \,, &&\text{as $\sigma_{\neg}=\sigma_{\neg}^{-1}$.}\\
&= \sigma_{\bp\bq}^{-k}\circ  \sigma_{\neg} \circ   (\sigma_{\bp} \circ \sigma_{\bq})^k   \circ \sigma_{\neg}\\
&= \sigma_{\bp\bq}^{-k}  \circ  (  \sigma_{\neg}  \circ \sigma_{\bp} \circ \sigma_{\neg} \circ \sigma_{\neg} \circ \sigma_{\bq} \circ \sigma_{\neg} )^k \,, &&\text{as $\sigma_{\neg}=\sigma_{\neg}^{-1}$.}\\
&= \sigma_{\bp\bq}^{-k} \circ  (  \sigma_{\bp}^{-1} \circ  \sigma_{\bq}^{-1} )^k\,, & &\text{by Corollary~\ref{C:inversemotsym}}.\\
&= \sigma_{\bp\bq}^{-k} \circ  (  \sigma_{\bq} \circ  \sigma_{\bp} )^{-k}\\
&= \sigma_{\bp\bq}^{-k}  \circ \sigma_{\bp\bq}^{-k} \,, &&\text{by \eqref{E:defext3}.}\\
&= \sigma_{\bp\bq}^{-2k}
\end{align*}
The result follows using an induction.
\end{proof}

We now fix some notation for various words that will be used many times in the sequel.

\begin{notation}\label{N:constructionmots}
We define $Q\colon A^*\to A^*$, inductively by $Q(p)=p$, and $Q(\bq p)=Q(\bq)pQ(\bq)$ where $\bq\in A^*$ and $p\in A$.

Given $a\in T$, we set $C(a)=C_\blank^a$. We extend the notation to all words $\ba=a_1\dots a_n$ in $T^*$ by $C(a_1a_2\dots a_n)=C(a_1)C(a_2)\dots C(a_n)$.

Set $Z(a_1a_2\dots a_n,\varepsilon)=(a_1,\varepsilon)(a_2,\varepsilon)\dots(a_n,\varepsilon)$

Set $P(a_1a_2\dots a_n)=P_{a_1}P_{a_2}\dots P_{a_n}$.
\end{notation}

Note that the definition of $Q$ is equivalent to putting $Q(p_1\dots p_n)=p_{i_1}p_{i_2}\dots p_{i_{2^{m}-1}}$, where $i_k$ is the position of the first $0$ that turns to~$1$ when one goes from $k-1$ to $k$ in binary, assuming that the digit with lowest weight is on the left. Equivalently, $i_k$ corresponds to the position where the carry is cleared when adding 1. For instance, counting $000, 100, 010, 110, 001, 101, 011, 111$, we find $Q(p_1p_2p_3)=p_1p_2p_1p_3p_1p_2p_1$. This definition is not essential in the paper, but it might help to understand the way counters work, and the construction of our elements in the group.

We obtain a short description of the inverse of $Q(\bp)$.

\begin{lemma}\label{L:inverseQuneg}
For every $\bp\in A^*$, we have $Q(\bp\neg)=Q(\bp)\neg Q(\bp)\equiv\neg$.
\end{lemma}

\begin{proof}
The result follows from Lemma~\ref{L:inversestate} and $Q(\bp)$ being a palindrome.
\end{proof}

\begin{definition}
A word $\bp\in A^*$ is \emph{regular} if we can write $\bp=\bq\br$ where $\br$ has only modifying states and $\bq$ has only non-counting states. The \emph{potential} of $\bp$ is $\infty$ if $\neg$ appears in $\br$, otherwise the \emph{potential} of $\bp$ is the length of $\br$.
\end{definition}

If we assume that $\bq$ does not end with $\neg$, then a decomposition $\bp=\bq\br$ as above is unique.

For each word~$\bx$, we now introduce a map~$\rho_{\bx}$ on the set of regular words. Let $\bp=p_1\dots p_{m+n}$ be a regular word, such that $p_1\dots p_m$ are non-counting states, $p_m\not=\neg$, and $p_{m+1}\dots p_{m+n}$ are modifying states. 
For~$x$ in $\Sigma$, we set
$$\rho_x(\bp)=\begin{cases}
\rho_x(\bp)=\delta_x(p_1)\delta_x(p_2)\dots\delta_x(p_m)
&\text{for $n = 0$},\\
\delta_x(p_1)\delta_x(p_2)\dots\delta_x(p_m)\delta_x(p_{m+1})\delta_{x^\bot}(p_{m+2})\dots \delta_{x^\bot}(p_{m+n})
&\text{for $n \ge1$}.\\
\end{cases}$$
Then, for $\bx=x_1\dots x_k$ in $\Sigma^*$, we set $\rho_{\bx}=\rho_{x_k}\circ\rho_{x_{k-1}}\circ\dots\circ\rho_{x_1}$. 

We give a few basic properties of regular words and of $\rho$.

\begin{lemma}\label{L:basicrho}
Let $\bp\in A^*$ and $q\in A$. Then the following statements hold.
\begin{enumerate}
\item If $\bp$ is regular with potential zero, then $\bp q$ is regular, and
\[
\rho_x(\bp q)=\rho_x(\bp)\delta_x(q)\,. 
\]

\item If $\bp$ is regular with potential at least one, and $q\in A$ is modifying, then $\bp q$ is regular with potential at least two, and
\[
\rho_x(\bp q)=\rho_x(\bp)\delta_{x^\bot}(q)\,. 
\]

\item If $\bp$ is regular with potential zero and $a\in T\setminus F$, then $\rho_{(a,0)}(\bp)$ is regular with potential zero.
\end{enumerate}
\end{lemma}

\begin{proof}
Points $(1)$ and $(2)$ are direct consequences of the definition of $\rho_x$.

Assume that $\bp=p_1\dots p_n$ has potential zero, hence $p_n$ is a non-modifying state and $p_1\dots p_{n-1}$ are non-counting states. Let $a\in T\setminus F$. We have 
$$\rho_{(a,0)}(\bp)=\delta_{(a,0)}(p_1)\dots\delta_{(a,0)}(p_n).$$
From Remark~\ref{R:faitbasesurdelta}(2), $\delta_{(a,0)}(p_1)\dots\delta_{(a,0)}(p_{n-1})$ are non-counting and, by Remark~\ref{R:faitbasesurdelta}(5), the state $\delta_{(a,0)}(p_{n})$ is non-modifying, hence $\rho_{(a,0)}(\bp)$ is regular with potential zero.
\end{proof}

We now describe the action of $Q(\bp)$ on a single letter.

\begin{lemma}\label{L:caracterisationQneg}
Let $\bp\in A^*$ of length at least one, and $x\in \Sigma$. Then $\sigma_{Q(\bp)}(x)=x^\bot$ if, and only if, the last letter of $\bp$ is a modifying state.
\end{lemma}

\begin{proof}
Write $\bp=p_1\dots p_m$. Note that $Q(\bp)=Q(p_1\dots p_{m-1})p_m Q(p_1\dots p_{m-1})$. So the number of modifying states in $Q(\bp)$ is odd if, and only if, $p_m$ is a modifying state. Hence, by Lemma~\ref{L:modifyingword}, we have $\sigma_{Q(\bp)}(x)=x^\bot$ if, and only if, $p_m$ is a modifying state.
\end{proof}

We can now specify the cross diagram of a regular word with potential zero completely.

\begin{lemma}\label{L:motsanscompteur}
Let $\bp=p_1\dots p_m\in A^*$, be without counting states. Let $x\in \Sigma$. Then the following statements hold.
\begin{enumerate}
\item $\delta_{x^\bot}(\bp)=\delta_{x}(\bp)=\delta_{x}(p_1)\delta_{x}(p_2)\dots \delta_{x}(p_m)=\rho_x(\bp)=\rho_{x^\bot}(\bp)$.
\item If $x\in T\times\set{0,1}$, then there is no counting state in $\delta_{x}(\bp)$.
\item $\delta_x(Q(\bp))=\delta_{x^\bot}(Q(\bp)) = Q(\rho_{x}(\bp))$.
\end{enumerate}
If $\bp$ is a regular word with potential zero, then we have the following cross diagram
\begin{equation}\label{E:CDsanssanscompteur}
 \xymatrix{
	& x \ar[dd] &\\
    Q(\bp) \ar[rr]  & &  Q(\rho_x(\bp))\\
     & x
  }
\end{equation}
\end{lemma}

\begin{proof}
Let $x\in\Sigma$. As observed in Remark~\ref{R:faitbasesurdelta}(1), we have $\delta_{x}(p)=\delta_{x^\bot}(p)$ for every non-counting state $p$. Ttherefore, for every $\bp=p_1\dots p_m$ without counting states, we have
\[
\delta_{x}(p_1)\dots \delta_{x}(p_m)=\delta_{x^\bot}(p_1)\dots \delta_{x^\bot}(p_m) = \rho_x(p_1\dots p_m)=\rho_{x^\bot}(p_1\dots p_m)\,.
\]

Assume that $(1)$ holds for a given $\bp=p_1\dots p_m$. Let $q\in A$ be a non-counting state. Either $q=\neg$, then $\sigma_q(x)=x^\bot$, or $q$ is not modifying, then $\sigma_q(x)=x$. In both cases, we find $\delta_{\sigma_q(x)}(\bp)=\delta_{x}(\bp)$. Therefore, the following equalities hold
\[
\delta_{x}(q\bp) = \delta_{x}(q) \delta_{\sigma_q(x)}(\bp)= \delta_{x}(q) \delta_{x}(\bp) = \delta_{x}(q)\delta_{x}(p_1)\delta_{x}(p_2)\dots \delta_{x}(p_m)\,.
\]
Moreover, we have $\delta_x(q)=\delta_{x^\bot}(q)$, hence 
\[
\delta_{x^\bot}(q\bp) = \delta_{x^\bot}(q)\delta_{\sigma_q(x^\bot)}(\bp)=  \delta_x(q)\delta_{x}(p_1)\delta_{x}(p_2)\dots \delta_{x}(p_m)\,.
\]
Therefore, we find
\[
\delta_{x^\bot}(q\bp)=\delta_{x}(q\bp)=\delta_{x}(q)\delta_{x}(p_1)\dots \delta_{x}(p_m)\,.
\]
Then, $(1)$ follows using an induction.

Assume that $x\in T\times\set{0,1}$. Note that $\delta_{x}(\bp) = \delta_{x}(p_1)\dots \delta_{x}(p_m)$. So, using Remark~\ref{R:faitbasesurdelta}(2), we see that no state in $\delta_{x}(\bp)$ is counting, thus $(2)$ holds.

Assume that $\delta_x(Q(\bp)) = Q(\rho_{x}(\bp))$ for some $\bp$. Let $q\in A$ be a non-counting state. The following equalities hold
\begin{align*}
\delta_x(Q(\bp q)) &= \delta_x(Q(\bp)qQ(\bp)) && \text{by definition of $Q$,}\\
& = \delta_x(Q(\bp))\delta_x(q) \delta_x(Q(\bp)) && \text{by $(1)$ for the word $Q(\bp)qQ(\bp)$,}\\
& = Q(\rho_x(\bp))\delta_x(q)Q(\rho_x(\bp))&& \text{by induction hypothesis,}\\
&= Q(\rho_x(\bp)\delta_x(q)) && \text{by definition of $Q$,}\\
& = Q(\rho_x(\bp q)) && \text{by definition of $\rho_x$.}
\end{align*}
Statement $(3)$ follows using an induction. If $\bp$ is regular with potential zero, then the last letter of $\bp$ is non-modifying. Then, Lemma~\ref{L:caracterisationQneg} implies that $\sigma_{Q(\bp)}(x)=x$, hence the cross diagram is as indicated in~\eqref{E:CDsanssanscompteur}.
\end{proof}

We now show that the action of $\rho_{(a,0)}$  on a regular word either reduces the potential by one, or it turns the latter to $\infty$.

\begin{lemma}\label{L:motregulierappliquerrho}
Let $\bp\in A^*$ be a regular word with potential $n\ge 1$. Let $a\in T$. Then $\rho_{(a,0)}(\bp)$ is regular with potential either $n-1$ or $\infty$.
\end{lemma}

\begin{proof}
Write $\bp=p_1\dots p_m$. Pick $k$ such that $p_1\dots p_k$ are not counting, $p_k\not=\neg$, and $p_{k+1}\dots p_{m}$ are modifying. The following equality holds
\[
\rho_{(a,0)}(\bp)=\delta_{(a,0)}(p_1)\dots\delta_{(a,0)}(p_k)\delta_{(a,0)}(p_{k+1})\delta_{(a,1)}(p_{k+2})\dots \delta_{(a,1)}(p_m)\,.
\]
By Remark~\ref{R:faitbasesurdelta}(3), the states $\delta_{(a,0)}(p_1)\dots\delta_{(a,0)}(p_k)\delta_{(a,0)}(p_{k+1})$ are not counting. By Remark~\ref{R:faitbasesurdelta}(4), the states $\delta_{(a,1)}(p_{k+2})\dots \delta_{(a,1)}(p_m)$ are modifying, hence $\rho_{(a,0)}(\bp)$ is regular. If $\delta_{(a,0)}(p_{k+1})=\neg$, then $\rho_{(a,0)}(\bp)$ is regular with potential $\infty$. Similarly, if some letter in $\delta_{(a,1)}(p_{k+2})\dots \delta_{(a,1)}(p_{m})$ is~$\neg$, then $\rho_{(a,0)}(\bp)$ has potential $\infty$.

Assume that $\delta_{(a,0)}(p_{k+1})\not=\neg$, and that no letter in $\delta_{(a,1)}(p_{k+2})\dots \delta_{(a,1)}(p_{m})$ is~$\neg$. Thus $\rho_{(a,0)}(\bp)$ is regular with potential $m-k-1$. We note that there is no letter $\neg$ in $p_{k+1}\dots p_m$. Hence the potential of $\bp$ is $n=m-k$, and the potential of $\rho_{(a,0)}(\bp)$ is $m-k-1=n-1$.
\end{proof}

We can now specify the cross diagram of a regular word with potential at least one completely.

\begin{lemma}\label{L:motregulierCompatibiliteQ}
Let $\bp\in A^*$ be a regular word with potential at least one. Let $x\in\Sigma$. Then $\delta_x(Q(\bp))=Q(\rho_x(\bp))$. In particular we have the following cross diagram
\begin{equation}\label{E:CDavecpotentiel}
 \xymatrix{
	& x \ar[dd] &\\
    Q(\bp) \ar[rr]  & &  Q(\rho_x(\bp))\\
     & x^\bot
  }
\end{equation}
\end{lemma}

\begin{proof}
Let $\bq\in A^*$ be a regular word with potential zero. Let $p\in A$ be a modifying state. Then we have the following cross diagram
\[
 \xymatrix{
	& \quad x\quad   \ar[dd] &\\
    Q(\bq) \ar[rr]  & \ins{\eqref{E:CDsanssanscompteur}} &  Q(\rho_x(\bq))\\
    & x\ar[dd]\\
	 p \ar[rr]  & & \delta_x(p)\\
	& x^\bot\ar[dd]\\
    Q(\bq) \ar[rr] & \ins{\eqref{E:CDsanssanscompteur}} &  Q(\rho_x(\bq))\\
	& x^\bot
  }
\]
We deduce
\begin{align*}
\delta_x(Q(\bq p)) =  \delta_x(Q(\bq)pQ(\bq))) &= Q(\rho_x(\bq))\delta_x(p)Q(\rho_x(\bq))\\
&= Q(\rho_x(\bq)\delta_x(p))=Q(\rho_x(\bq p)).
\end{align*}

Let $\bp\in A^*$ be a regular word with potential at least one. Assume $\delta_x(Q(\bp))=Q(\rho_x(\bp))$. As the last letter of $\bp$ is modifying, Lemma~\ref{L:caracterisationQneg} implies $\sigma_{Q(\bp)}(x)=x^\bot$. Next, if $r\in A$ is a modifying state, we find $\sigma_r(x^\bot)=x^{\bot\bot}=x$. The above equalities can be summarized in the following cross diagram.
\[
 \xymatrix{
	& x \ar[dd] &\\
    Q(\bp) \ar[rr]  & &  Q(\rho_x(\bp))\\
    & x^\bot\ar[dd]\\
	r \ar[rr]  & & \delta_{x^\bot}(r)\\
	& x\ar[dd]\\
	Q(\bp) \ar[rr]  & &  Q(\rho_x(\bp))\\
	& x^\bot
  }
\]
Therefore, we find
\begin{align*}
\delta_x(Q(\bp r))=\delta_x(Q(\bp)rQ(\bp)) &= Q(\rho_x(\bp))\delta_{x^\bot}(r)Q(\rho_x(\bp))\\
&= Q(\rho_x(\bp) \delta_{x^\bot}(r))=Q(\rho_x(\bp r))\,.
\end{align*}
An induction then gives $\delta_x(Q(\bp))=Q(\rho_x(\bp))$ for all regular $\bp\in A^*$ with potential at least one.
\end{proof}

The next step is to prove that a regular word whose potential~$n$ is large enough transforms at least $n$~digits~$0$ to 1. We refer to Notation~\ref{N:constructionmots} for the definition of $Z(\ba,0)$ and $Z(\ba,1)$.

\begin{lemma}\label{L:regularcompose}
Let $\ba\in T^*$ of length $m$. Let $\bp\in A^*$ regular with potential $n\ge m$. Then we have the following cross diagram
\begin{equation}\label{E:potentielconvertir}
 \xymatrix{
	& Z(\ba,0)\ar[dd] &\\
    Q(\bp) \ar[rr]  & &  Q(\rho_{Z(\ba,0)}(\bp))\\
    & Z(\ba,1)
  }
\end{equation}
Moreover $\rho_{Z(\ba,0)}(\bp)$ is regular with potential either $n-m$ or $\infty$.
\end{lemma}

\begin{proof}
Assume $\ba$ is a length~$m$ word of~$T^*$ and the property of the lemma holds for all $\bp\in A^*$ with potential $n\ge m$. Let $\bp$ be a word with potential $n\ge m+1$. Then, the word $\rho_{Z(\ba,0)}(\bp)$ is regular with potential either $n-m$ or $\infty$, in particular $\rho_{Z(\ba,0)}(\bp)$ has potential at least one.

Let $b\in T$. Then we have $\rho_{Z(\ba b,0)}(\bp)=\rho_{(b,0)}(\rho_{Z(\ba,0)}(\bp))$. It follows from Lemma~\ref{L:motregulierappliquerrho} that $\rho_{Z(\ba b,0)}(\bp)$ is a regular word with potential either $n-m-1$ or $\infty$. We have the following cross diagram, where the first cross is the induction hypothesis.
\[
 \xymatrix{
	& Z(\ba,0)\ar[dd] & & (b,0)\ar[dd]\\
    Q(\bp) \ar[rr]  &  &  Q(\rho_{Z(\ba,0)}(\bp)) \ar[rr] &  \ins{\eqref{E:CDavecpotentiel}} & Q(\rho_{(b,0)}(\rho_{Z(\ba,0)}(\bp))) = Q(\rho_{Z(\ba b,0)}(\bp)) \\
    & Z(\ba,1) & & (b,1) 
  }
\]
The result follows using an induction.
\end{proof}

\begin{lemma}\label{L:regularchainelongue}
Let $\ba\in (T\setminus F)^*$. Let $\bp\in A^*$ regular. Assume that the potential of $\rho_{Z(\ba,0)}(\bp)$ is at least one. Then we have the following cross diagram
\begin{equation}\label{E:potentielconvertirLong}
 \xymatrix{
	& Z(\ba,0)\ar[dd] &\\
    Q(\bp) \ar[rr]  & &  Q(\rho_{Z(\ba,0)}(\bp))\\
    & Z(\ba,1)
  }
\end{equation}
\end{lemma}

\begin{proof}
Denote by $m$ the length of $\ba$ and by $n$ the potential of $\bp$. If $n\ge m$, then \eqref{E:potentielconvertirLong} is a special case of \eqref{E:potentielconvertir}. Assume $n<m$. Lemma~\ref{L:regularcompose} implies that the potential of $\rho_ {(a_1,0)\dots(a_n,0)}(\bp)$ is either $0$ or $\infty$. However, if it is $0$, then Lemma~\ref{L:basicrho}(3) implies that $\rho_{Z(\ba,0)}(\bp)=\rho_{(a_{n+1},0)\dots(a_m,0)}(\rho_{(a_1,0)\dots(a_n,0)}(\bp))$ has potential 0, a contradiction. Therefore, $\rho_{(a_1,0)\dots(a_n,0)}(\bp)$ has potential $\infty$. Hence, we have the following cross diagram
\[
 \xymatrix{
	& (a_1,0)\dots (a_n,0) \ar[dd] & & (a_{n+1},0)\dots (a_m,0) \ar[dd]\\
    Q(\bp) \ar[rr]  &  \ins{\eqref{E:potentielconvertir}} &   Q(\rho_{(a_1,0)\dots (a_n,0)}(\bp)) \ar[rr]  &  \ins{\eqref{E:potentielconvertir}} &  Q(\rho_{Z(\ba,0)}(\bp))\\
    & (a_1,1)\dots (a_n,1) & & (a_{n+1},1)\dots (a_m,1)
  }
\]
as announced.
\end{proof}

\section{A lower bound on the order}

The goal of this section is to show that an element of our automaton group simulates the cellular automaton when acting on the appropriate words. We start with a word $\ba=a_1\dots a_n\in (T\setminus F)^*$, and consider the word $\bb=b_1\dots b_n b_{n+1}$ constructed by applying one step of the cellular automaton to $\ba$, and completed with the state $\blank$, namely the word~$\bb$ satisfying
\begin{gather*}
b_k=t(a_{k-1},a_k,a_{k+1}) \qquad \text{for $2\le k\le n-1$,}\\ 
b_1=t(\blank,a_1,a_2), \qquad b_n=t(a_{n-1},a_n,\blank), \qquad \text{and} \qquad b_{n+1}=t(a_n,\blank,\blank). 
\end{gather*}
Our aim is to establish the correctness of the following cross diagram
\begin{equation}\label{E:CrossDiag}
 \xymatrix{
	& Z(\ba\$,0)\ar[dd] &\\
    (Q(C(\ba))C(\$))^2 \ar[rr]  & &  Q(C(\bb))C(\$)\\
    & Z(\ba\$,0)
  }
\end{equation}
The element $Q(C(\ba))C(\$)$ of our automaton group encodes the configuration~$\ba$. The cross diagram connects the square of $Q(C(\ba))C(\$)$ with the element~$Q(C(\bb))C(\$)$ encoding the next configuration. Thus, squaring the element of the group corresponds to one step of computation in the cellular automaton---and this is the point of our construction.

We begin with a description of the first part of our cross diagram. We recall from Notation~\ref{N:applyingcellularautomata} that, in the above situation, we write $\bb = \tau(\ba)$.

\begin{lemma}\label{L:deltaZareg}
Let $n\ge 1$ and $\ba=a_1\dots a_n\in (T\setminus F)^n$. Set $a_0=a_n=a_{n+1}=\blank$, and consider $b_1\dots b_{n+1}=\tau(\ba)$. Then, for every $1\le k\le n$, we have
\[
\rho_{(a_1,0)(a_2,0)\dots(a_k,0)}(C_e^{a_1}\dots C_e^{a_n})=P_{b_1}\dots P_{b_{k-1}}W_{a_{k-1}}^{a_k} C_{a_k}^{a_{k+1}}\dots C_{a_k}^{a_{n}}\,.
\]
\end{lemma}

\begin{proof}
Owing to the definition of $\rho$, \eqref{E:deltaC4} and \eqref{E:deltaW2} imply
\[
\rho_{(a_1,0)}(C_e^{a_1}\dots C_e^{a_n}) = \delta_{(a_1,0)}(C_e^{a_1}) \delta_{(a_1,1)}(C_e^{a_2})\dots\delta_{(a_1,1)}(C_e^{a_n})= W_{a_0}^{a_1} C_{a_1}^{a_{2}}\dots C_{a_1}^{a_{n}}\,.
\]
For an induction,  assume that the formula is true for some~$k < n$. Then the following equalities hold
\begin{align*}
&\rho_{(a_1,0)(a_2,0)\dots(a_{k+1},0)}(C_e^{a_1}\dots C_e^{a_n}) \\
&\quad=\rho_{(a_{k+1},0)}(\rho_{(a_1,0)(a_2,0)\dots(a_{k},0)}(C_e^{a_1}\dots C_e^{a_n})) & & \text{by definition $\rho$,}\\
&\quad=\rho_{(a_{k+1},0)}(P_{b_1}\dots P_{b_{k-1}}W_{a_{k-1}}^{a_k} C_{a_k}^{a_{k+1}}\dots C_{a_k}^{a_{n}}) & & \text{by induction hypothesis,} \\
&\quad=P_{b_1}\dots P_{b_{k-1}} \delta_{(a_{k+1},0)}(W_{a_{k-1}}^{a_k})  \delta_{(a_{k+1},0)} (C_{a_k}^{a_{k+1}}) 
\rlap{$\delta_{(a_{k+1},1)}(C_{a_k}^{a_{k+2}})\dots \delta_{(a_{k+1},1)}(C_{a_k}^{a_{n}})$}\\
&  & & \text{by definition of $\rho$ and \eqref{E:deltaP2},}\\
&\quad= P_{b_1}\dots P_{b_{k-1}} P_{b_k} W_{a_k}^{a_{k+1}}C_{a_{k+1}}^{a_{k+2}}\dots C_{a_{k+1}}^{a_{n}} & & \text{by \eqref{E:deltaC4} and \eqref{E:deltaW2}.}
\end{align*}
The result then follows using an induction.
\end{proof}

We now investigate what happens when there is a discrepancy between the ``ideal" tape and the given tape: one of the state is transformed to $\neg$, yielding a regular word with infinite potential.

\begin{corollary}\label{C:erreurpotentielinfini}
Let $a_1\dots a_n\in T^*$ and $c_1\dots c_N\in T^*$. If there is $i\le\min(n,N)$ such that $c_i\not=a_i$ or $c_i\in F$, then $\rho_{(c_1,0)(c_2,0)\dots(c_{N},0)}(C_e^{a_1}\dots C_e^{a_n})$ is regular with potential $\infty$.
\end{corollary}

\begin{proof}
For $c_i\in F$, \eqref{E:deltaNegouF} implies that, for every $\bp=p_1\dots p_n\in A^n$, we have $\rho_{(c_i,0)}(p_1\dots p_n)=\neg\dots\neg$, and, therefore, $\rho_{(c_1,0)(c_2,0)\dots(c_{N},0)}(C_e^{a_1}\dots C_e^{a_n})=\neg\dots\neg$ is regular with potential $\infty$. We can now assume that $c_1\dots c_N\in (T\setminus F)^*$.

If $c_1\not=a_1$, we find
\begin{align*}
\rho_{(c_1,0)}(C_e^{a_1}\dots C_e^{a_n}) &= \delta_{(c_1,0)}(C_e^{a_1}) \delta_{(c_1,1)}(C_e^{a_2})\dots\delta_{(c_1,1)}(C_e^{a_n}) & & \text{by definition of $\rho$.}\\
&=\neg C_{c_1}^{a_2}\dots C_{c_1}^{a_n} & & \text{by \eqref{E:deltaC4} and  \eqref{E:deltaC5}}\,.
\end{align*}
Hence, $\rho_{(c_1,0)}(C_e^{a_1}\dots C_e^{a_n})$ is regular with potential $\infty$. Then, Lemma~\ref{L:regularcompose} implies that $\rho_{(c_1,0)\dots (c_N,0)}(C_e^{a_1}\dots C_e^{a_n})$ is regular with potential $\infty$.

Assume $c_1=a_1$. Let $b_1,\dots b_n$ be as in Lemma~\ref{L:deltaZareg}. Pick $k$ minimal such that $c_{k+1}\not=a_{k+1}$. Then, for every $j\le k$, we have $c_j=a_j$. So Lemma~\ref{L:deltaZareg} implies
\begin{equation}\label{E:equat1Erreurpotentielinfini}
\rho_{(c_1,0)(c_2,0)\dots(c_k,0)}(C_e^{a_1}\dots C_e^{a_n})=P_{b_1}\dots P_{b_{k-1}}W_{a_{k-1}}^{a_k} C_{a_k}^{a_{k+1}}\dots C_{a_k}^{a_{n}}\,.
\end{equation}
We deduce
\begin{align*}
&\rho_{(c_1,0)\dots (c_k,0)(c_{k+1},0)}(C_e^{a_1}\dots C_e^{a_n})\\
&\quad=\rho_{(c_{k+1},0)}(\rho_{(c_1,0)\dots (c_k,0)}(C_e^{a_1}\dots C_e^{a_n})) & & \text{by definition of $\rho$,}\\
&\quad=\rho_{(c_{k+1},0)}(P_{b_1}\dots P_{b_{k-1}}W_{a_{k-1}}^{a_k} C_{a_k}^{a_{k+1}}\dots C_{a_k}^{a_{n}}) & & \text{by \eqref{E:equat1Erreurpotentielinfini},}\\
&\quad=P_{b_1}\dots P_{b_{k-1}} \delta_{(c_{k+1},0)}(W_{a_{k-1}}^{a_k}) \delta_{(c_{k+1},0)}(C_{a_k}^{a_{k+1}}) \rlap{$\delta_{(c_{k+1},1)}(C_{a_k}^{a_{k+2}})\dots \delta_{(c_{k+1},1)}(C_{a_k}^{a_{n}})$}\\
& & &\text{by definition of $\rho$ and \eqref{E:deltaP2},}\\
&\quad=P_{b_1}\dots P_{b_{k-1}} P_{t(a_{k-1},a_k,c_{k+1})} \neg C_{c_{k+1}}^{a_{k+2}}\dots C_{c_{k+1}}^{a_{n}} & & \text{by \eqref{E:deltaC4}, \eqref{E:deltaC5}, and \eqref{E:deltaW2}.}
\end{align*}
Therefore $\rho_{(c_1,0)\dots (c_k,0)(c_{k+1},0)}(C_e^{a_1}\dots C_e^{a_n})$ is regular with potential $\infty$. Finally, Lemma~\ref{L:regularcompose} implies that $\rho_{(c_1,0)\dots (c_N,0)}(C_e^{a_1}\dots C_e^{a_n})$ is regular with potential $\infty$.
\end{proof}

We now make the next step toward the cross diagram of~\eqref{E:CrossDiag}.

\begin{lemma}\label{L:deltaZa}
Let $n\ge 1$ and $\ba=a_1\dots a_n\in (T\setminus F)^n$. Set $a_0=a_{n+1}=a_{n+2}=\blank$, and consider $b_1\dots b_{n+1}=\tau(\ba)$. Then, the following cross diagram is valid
\begin{equation}\label{E:CDsansDolarnierreur}
 \xymatrix{
	& Z(\ba,0)\ar[dd] &\\
    Q(C(\ba)) \ar[rr]  & &  Q(P(b_1\dots b_{n-1}) W_{a_{n-1}}^{a_n})\\
    & Z(\ba,1)
  }
\end{equation}
Moreover, for every $\varepsilon\in\set{0,1}$, the following cross diagram is valid
\begin{equation}\label{E:CDpassagedollar}
 \xymatrix{
	& (\$,\varepsilon)\ar[dd] &\\
    Q(P(b_1\dots b_{n-1}) W_{a_{n-1}}^{a_n}) \ar[rr]  & & Q(C(b_1\dots b_n)) \\
    & (\$,\varepsilon)
  }
\end{equation}
\end{lemma}

\begin{proof}
Lemma~\ref{L:deltaZareg} implies $\rho_{Z(\ba,0)}(C(\ba)) =  P(b_1\dots b_{n-1}) W_{a_{n-1}}^{a_n}$. As $C(\ba)$ is regular with potential $n$, the cross diagram \eqref{E:CDsansDolarnierreur} follows from \eqref{E:potentielconvertir}.

Next, \eqref{E:deltaP1} implies $\delta_{(\$,\varepsilon)}(P_{b_k})=C_\blank^{b_k}$ for every $k\le n-1$, and \eqref{E:deltaW1} implies $\delta_{(\$,\varepsilon)}(W_{a_{n-1}}^{a_n})=C_\blank^{t(a_{n-1},a_n,\blank)}=C_\blank^{b_n}$, whence $\rho_{(\$,\varepsilon)}(P(b_1\dots b_{n-1}) W_{a_{n-1}}^{a_n}) = C_e^{b_1}\dots C_e^{b_n}$. Hence, \eqref{E:CDpassagedollar} follows from \eqref{E:CDsanssanscompteur}.
\end{proof}

We now check that the element we are constructing is not the identity.

\begin{lemma}\label{L:ordrefindeliste}
Let $\ba\in T^*$. Assume that there is $b\in T\setminus F$ satisfying $t(b,e,e)\not\in F$. Then the order of $Q(C(\ba))C(\$)$ is even or infinite in~$G$.
\end{lemma}

\begin{proof}
Let $n\ge 0$, and $\ba=a_1\dots a_n\in T^n$. Assume that $n=0$. Then, as $\sigma_{C(\$)}((\blank,0))=(\blank,1)$ and $\sigma_{C(\$)}^2((\blank,0))=(\blank,0)$, the order of $C(\$)$ is even or infinite in~$G$. We hereafter assume that $n\ge 1$.

By \eqref{E:deltaC5} and the definition of $Q$, we have $\delta_{(\$,0)}(Q(C(\ba)))=\neg\dots\neg\equiv\neg$, since the number of $\neg$ symbols is odd. Then, the following cross diagram is valid
\[
 \xymatrix{
	& (\$,0)\ar[dd] & & (b,0) \ar[dd] & & (\$,1) \ar[dd] & & (t(b,\blank,\blank),1)\ar[dd] & & (b,0)\ar[dd]\\
    Q(C(\ba)) \ar[rr]  &  &  \neg \ar[rr] & \ins{\eqref{E:deltaNegouF}} & \neg \ar[rr] & \ins{\eqref{E:deltaNegouF}} & \neg \ar[rr] & \ins{\eqref{E:deltaNegouF}} & \neg \ar[rr] & \ins{\eqref{E:deltaNegouF}} & \neg \\
	& (\$,1)\ar[dd] & & (b,1) \ar[dd] & & (\$,0)\ar[dd] &&  (t(b,\blank,\blank),0)\ar[dd] & & (b,1) \ar[dd]\\
    C_\blank^{\$} \ar[rr] & \ins{\eqref{E:deltaC1}} & C_\blank^{\$} \ar[rr]  & \ins{\eqref{E:deltaC4}} & C_b^{\$}\ar[rr]& \ins{\eqref{E:deltaC3}} &  C_\blank^{t(b,\blank,\blank)}  \ar[rr]  & \ins{\eqref{E:deltaC6}} & W_e^{t(b,e,e)} \ar[rr] & \ins{\eqref{E:deltaW2}} &  P_{t(e,t(b,e,e),b)}  \\
    & (\$,0) && (b,0) & & (\$,1) && (t(b,\blank,\blank),1) && (b,1)
  }
\]
Set $\bx=(\$,0)(b,0)(\$,1)(t(b,\blank,\blank),1)(b,0)$. We find $\sigma_{Q(C(\ba))C(\$)} (\bx)\not=\bx$ and, moreover, $\sigma_{Q(C(\ba))C(\$)}^2 (\bx)=\bx$. Therefore, the order of ${Q(C(\ba))C(\$)}$ is even or infinite in~$G$.
\end{proof}

We can now connect the orders of two elements corresponding to successive configurations of the cellular automaton. 

\begin{lemma}\label{L:deltaZa2}
Let $\ba\in (T\setminus F)^*$. Then the following cross diagram is valid
\begin{equation}\label{CD:multiplierordre}
 \xymatrix{
	& Z(\ba\$,0)\ar[dd] &\\
    (Q(C(\ba))C(\$))^2 \ar[rr]  & &  Q(C(\tau(\ba)))C(\$)\\
    & Z(\ba\$,0)
  }
\end{equation}
Moreover, if there is $b\in T\setminus F$ such that $t(b,e,e)\not\in F$, then the order of $Q(C(\ba))C(\$)$ is at least twice the order of $Q(C(\tau(\ba)))C(\$)$ in~$G$.
\end{lemma}

\begin{proof}
Let $\ba=a_1\dots a_n\in (T\setminus F)^n$. Set $a_0=a_{n+1}=a_{n+2}=\blank$ and $\bb=b_1\dots b_{n+1}=\tau(\ba)$. Then, the following cross diagrams are valid
\[
 \xymatrix{
	& (a_1,0)\dots(a_n,0)\ar[dd] & &(\$,0)\ar[dd] \\
    Q(C(a_1\dots a_n)) \ar[rr]  & \ins{\eqref{E:CDsansDolarnierreur}} &  Q(P(b_1\dots b_{n-1}) W_{a_{n-1}}^{a_{n}}) \ar[rr] & \ins{\eqref{E:CDpassagedollar}} &  Q(C(b_1\dots b_{n-1}b_n)) \\
    & (a_1,1)\dots(a_n,1) \ar[dd]& & (\$,0)\ar[dd]\\
	C_{\blank}^{\$} \ar[rr] &  \ins{\eqref{E:deltaC4}} & C_{a_n}^{\$} \ar[rr] & \ins{\eqref{E:deltaC3}} & C_{\blank}^{b_{n+1}}\\
	& (a_1,0)\dots(a_n,0)\ar[dd] & &(\$,1)\ar[dd] \\
    Q(C(a_1\dots a_n)) \ar[rr]  & \ins{\eqref{E:CDsansDolarnierreur}} &  Q(P(b_1\dots b_{n-1}) W_{a_{n-1}}^{a_{n}}) \ar[rr] & \ins{\eqref{E:CDpassagedollar}} &  Q(C(b_1\dots b_{n-1}b_n)) \\
	& (a_1,1)\dots(a_n,1)\ar[dd] & &(\$,1)\ar[dd] \\
	C_{\blank}^{\$} \ar[rr] & \ins{\eqref{E:deltaC4}} & C_{a_n}^{\$} \ar[rr] & \ins{\eqref{E:deltaC1}} & C_{\blank}^{\$}\\
	& (a_1,0)\dots(a_n,0)\ & &(\$,0)
  }
\]
Composing the above diagrams gives the one of~\eqref{CD:multiplierordre}.

Lemma~\ref{L:ordrefindeliste} implies that the order of $Q(C(\ba))C(\$)$ is infinite or even. We can assume that the order of $Q(C(\ba))C(\$)$ is $2k$ for some $k\ge 1$. Let $\by\in\Sigma^*$. Then we find
\begin{align*}
Z(\ba\$,0)\by &= \sigma_{Q(C(\ba))C(\$)}^{2k}(Z(\ba\$,0)\by) & & \text{as the order of ${Q(C(\ba))C(\$)}$ is $2k$,}\\
& = \sigma_{(Q(C(\ba))C(\$))^2}^{k}(Z(\ba\$,0)\by) & & \text{by \eqref{E:defext3},}\\
& = Z(\ba\$,0)\sigma_{Q(C(\bb))C(\$)}^{k}(\by)& & \text{by \eqref{CD:multiplierordre} applied $k$ times.}
\end{align*}
Therefore $\sigma_{Q(C(\bb))C(\$)}^{k}(\by)=\by$, hence the order of $Q(C(\bb))C(\$)$ is at most $k$. Thus the order of $Q(C(\ba))C(\$)$ is at least twice the order of $Q(C(\bb))C(\$)$.
\end{proof}

We conclude the section with a lower bound for the order of an element of our group, assuming that the cellular automaton does not stop in $k$ steps.

\begin{lemma}\label{L:minorerordre}
Let $\ba\in T^*$, and $k\ge 1$. Assume that $\ba,\tau(\ba),\dots,\tau^k(\ba)$ do not contain letters of~$F$. Then the order of ${Q(C(\ba))C(\$)}$ is at least $2^{k+2}$ in~$G$.
\end{lemma}

\begin{proof}
Let $\ba=a_1\dots a_n\in T^*$. The last letter of $\tau(\ba)$ is $t(a_n,\blank,\blank)$. So $a_n\not\in F$ and $t(a_n,\blank,\blank)\not\in F$. 

Let $\bb\in (T\setminus F)^*$. By Lemma~\ref{L:deltaZa2}, the order of $Q(C(\bb))C(\$)$ is at least twice the order of $Q(C(\tau(\bb)))C(\$)$. Moreover, by Lemma~\ref{L:ordrefindeliste}, the order of $Q(C(\tau(\bb)))C(\$)$ is at least two, therefore the order of $Q(C(\bb))C(\$)$ is at least 4.

Assume that $k\ge 0$ is such that, for every $\bb\in T^*$, if $\bb,\tau(\bb),\dots,\tau^k(\bb)$ do not contain letters of~$F$, then the order of $Q(C(\bb))C(\$)$ is at least $2^{k+2}$. Let $\bb\in T^*$ be such that $\bb,\tau(\bb),\dots,\tau^{k+1}(\bb)$ do not contain letters of~$F$. Then, $\tau(\bb),\tau(\tau(\bb))\dots,\tau^{k}(\tau(\bb))$ do not contain letters of~$F$, hence the order of $Q(C(\tau(\bb)))C(\$)$ is at least $2^{k+2}$. Therefore, it follows from Lemma~\ref{L:deltaZa2} that the order of $Q(C(\bb))C(\$)$ is at least $2\times 2^{k+2}$. The result follows by induction.
\end{proof}

\begin{remark}\label{R:motdegrandordre}
We can explicitly construct an element witnessing the lower bound on the order of ${Q(C(\tau^{k+1}(\ba)))C(\$)}$ in~$G$. Indeed, let $b$ be as in Lemma~\ref{L:ordrefindeliste}, and set
\[
\bx=Z(\ba\$ \tau(\ba)\$ \dots \$ \tau^k(\ba),0)(\$,0)(b,0)(\$,1)(t(b,\blank,\blank),1)(b,0)\,.
\]
It follows from the proof of Lemma~\ref{L:minorerordre} and from Lemma~\ref{L:ordrefindeliste} that, for every $i< 2^{k+2}$, we have $\sigma^i_{Q(C(\ba))C(\$)}(\bx)\not=\bx$.
\end{remark}

\section{An upper bound on the order}

In this section, we give an upper bound on the order of $Q(C(\ba))C(\$)$ in~$G$. The main idea is that the word constructed in Remark~\ref{R:motdegrandordre} is, up to small variations, the worst possible one.

The first step is to establish that, if $\bx$ is not the word of Remark~\ref{R:motdegrandordre}, and if the first discrepancy between $\bx$ and the word of Remark~\ref{R:motdegrandordre} occurs sufficiently close to the beginning, then the $g$-orbit of $\bx$ is small. Also note that $(1)$ or $(2)$ below hold when $a$ lies in~$F$. Therefore, these conditions can be used as an halting condition.

\begin{lemma}\label{L:ordrepetit}
Let $\ba=a_1\dots a_n\in T^*$ be of length $n$ and $\bx=x_1\dots x_N\in\Sigma^*$ be of length $N\ge n$.
Then the following statement holds
\begin{enumerate}
\item If there exists $ i\le n$ such that $x_i\in F\times \set{0,1}$, then $\sigma_{Q(C(\ba))C_{\blank}^{\$}}^2(\bx)=\bx$.
\item If there exists $ i\le n$ such that $x_i\not=(a_i,0)$, then $\sigma_{Q(C(\ba))C_{\blank}^{\$}}^2(\bx)=\bx$.
\item If $x_{n+1}\not=(\$,0)$ and $x_{n+1}\not=(\$,1)$, then $\sigma_{Q(C(\ba))C_{\blank}^{\$}}^4(\bx)=\bx$.
\end{enumerate}
\end{lemma}

\begin{proof}
We can assume that the last letter of~$\bx$ is $(\$,0)$. Let us write $x_k=(c_k,\varepsilon_k)$ with $c_k\in T'$ and $\varepsilon_k\in\set{0,1}$, for every $k\le N$.

\begin{claim}
Suppose there exists $1\le k\le N$ such that
\begin{enumerate}
\item[$(i)$] For every $ i<k$, we have $\varepsilon_i=0$, $c_i\not=\$$, and $c_i\not\in F$.
\item[$(ii)$] $c_k=\$$ or $\varepsilon_k=1$ or $c_k\in F$.
\item[$(iii)$] $\rho_{(c_1,0)(c_2,0)\dots(c_{k-1},0)}(C(\ba))$ is regular with potential at least one.
\end{enumerate}
Then $\sigma_{Q(C(\ba))C_{\blank}^{\$}}^2(\bx)=\bx$.
\end{claim}

\begin{proof}
Set $\bq=\rho_{(c_1,0)(c_2,0)\dots(c_{k-1},0)}(C(\ba))$ and $\br=\rho_{(c_k,\varepsilon_k)}(\bq)$. Assume that $c_k\in F$. Then the following cross diagram is valid
\[
 \xymatrix{
	& (c_1,0)\dots(c_{k-1},0)\ar[dd] & &(c_k,\varepsilon_k)\ar[dd] \\
    Q(C(a_1\dots a_n)) \ar[rr] & \ins{\eqref{E:potentielconvertirLong}} &  Q(\bq) \ar[rr] & \ins{\eqref{E:deltaNegouF}}  &  Q(\neg\dots\neg) \\
    & (c_1,1)\dots(c_{k-1},1) \ar[dd]& & (c_k,1-\varepsilon_k)\ar[dd]\\
	C_{\blank}^{\$} \ar[rr] & \ins{\eqref{E:deltaC4}}  & C_{c_{k-1}}^{\$} \ar[rr] & \ins{\eqref{E:deltaNegouF}}  & \neg\\
	& (c_1,0)\dots(c_{k-1},0) & &(c_k,\varepsilon_k)
  }
\]
As $Q(\neg\dots\neg)\neg\equiv\id$, we deduce $\sigma_{Q(C(\ba))C_e^{\$}}(\bx)=\bx$.

Assume first $c_k=\$$. From the definition of $\rho$, the last letter $s$ of $\bq$ is of the form $s=\delta_{\by}(C_\blank^{a_n})$ for some $\by\in \Sigma^{k-1}$. Using an induction, we deduce from Remark~\ref{R:faitbasesurdelta}(6) that $s$ is not in $\set{C}\times T\times\set{\$}$. Moreover, by $(iii)$, the word $\bq$ is regular with potential at least one. Hence, its last letter $s$ is modifying, so, by Remark~\ref{R:faitbasesurdelta}(7), we obtain $\delta_{(\$,\varepsilon_k)}(s)=\delta_{(\$,1-\varepsilon_k)}(s)=\neg$. Therefore the last letter of $\br=\rho_{(\$,\varepsilon_k)}(\bq)$ is~$\neg$. Moreover, we obtain the following cross diagram
\[
 \xymatrix{
	& (c_1,0)\dots(c_{k-1},0)\ar[dd] & &(\$,\varepsilon_k)\ar[dd] \\
    Q(C(a_1\dots a_n)) \ar[rr] & \ins{\eqref{E:potentielconvertirLong}} &  Q(\bq) \ar[rr] & \ins{\eqref{E:CDavecpotentiel}} &  Q(\br) \\
    & (c_1,1)\dots(c_{k-1},1) \ar[dd]& & (\$,1-\varepsilon_k)\ar[dd]\\
	C_{\blank}^{\$} \ar[rr] & \ins{\eqref{E:deltaC4}} & C_{c_{k-1}}^{\$} \ar[rr] & & \delta_{(\$,1-\varepsilon_k)}(C_{c_{k-1}}^{\$}) \\
	& (c_1,0)\dots(c_{k-1},0) & &(\$,\varepsilon_k)
  }
\]
As the last letter of $\br$ is $\neg$, Lemmas~\ref{L:inverseQuneg} and \ref{L:inversestate} imply the equivalences
\[
(Q(\br)\delta_{(\$,1-\varepsilon_k)}(C_{c_{k-1}}^{\$}))^2 \equiv (\neg \delta_{(\$,1-\varepsilon_k)}(C_{c_{k-1}}^{\$}))^2\equiv\id\,.
\]
So we have $\sigma_{Q(C(\ba))C_e^{\$}}^2(\bx)=\bx$.

Assume now $c_k\not=\$$ and $\varepsilon_k=1$. Then we find $\br=\rho_{(c_k,1)}(\bq)$, whence
\[
 \xymatrix{
	& (c_1,0)\dots(c_{k-1},0)\ar[dd] & &(c_k,1)\ar[dd] \\
    Q(C(a_1\dots a_n)) \ar[rr] & \ins{\eqref{E:potentielconvertirLong}}  &  Q(\bq) \ar[rr] &  \ins{\eqref{E:CDavecpotentiel}} &  Q(\br) \\
    & (c_1,1)\dots(c_{k-1},1) \ar[dd]& & (c_k,0)\ar[dd]\\
	C_{\blank}^{\$} \ar[rr] &  \ins{\eqref{E:deltaC4}} & C_{c_{k-1}}^{\$} \ar[rr] & \ins{\eqref{E:deltaC5}} & \neg\\
	& (c_1,0)\dots(c_{k-1},0) & &(c_k,1)
  }
\]
By Lemma~\ref{L:inverseQuneg}, we have $(Q(\br)\neg)^2\equiv\id$, and we deduce $\sigma_{Q(C(\ba))C_e^{\$}}^2(\bx)=\bx$. Thisestablishes Claim 1.
\end{proof}

Returning to the proof of Lemma~\ref{L:ordrepetit},  let $k$ be minimal such that $c_k=\$$, $\varepsilon_k=1$, or $c_k\in F$. Then Condition $(ii)$ of the claim holds. Moreover, by minimality of $k$, Condition $(i)$ holds as well. If $k\le n$, then, by Lemma~\ref{L:regularcompose}, the word $\rho_{(c_1,0)(c_2,0)\dots(c_{k-1},0)}(C(\ba))$ is regular with potential either $n-(k-1)\ge 1$ or $\infty$, hence Condition $(iii)$ holds. Then Claim~1 implies $\sigma_{Q(C(\ba))C_e^{\$}}^2(\bx)=\bx$. Thus $(1)$ holds.

We can now assume that $k\ge n+1$. Assume that there is $1\le i\le n$ such that $a_i\not=c_i$. Then, by Corollary~\ref{C:erreurpotentielinfini}, the word $\rho_{(c_1,0)(c_2,0)\dots(c_{k-1},0)}(C(\ba))$ is regular with potential $\infty$, so Condition $(iii)$ holds. It follows from Claim~1 that $(2)$ holds.

Finally, assume that $c_{n+1}\not=\$$. Set $\bq=\rho_{(a_1,0)\dots (a_n,0)}(C(a_1\dots a_n))$. Lemma~\ref{L:deltaZareg} implies $\bq= P_{b_1}\dots P_{b_{n-1}}W_{a_{n-1}}^{a_n}$, thus there is no counting state in $\bq$. Hence, by Lemma~\ref{L:motsanscompteur}(1), we have $\rho_{(c_{n+1},0)}(\bq) = \rho_{(c_{n+1},1)}(\bq)$. Set $\br=\rho_{(c_{n+1},0)}(\bq)$. We obtain the following cross diagram 
\[
 \xymatrix{
	& (a_1,0)\dots(a_{n},0)\ar[dd] & &(c_{n+1},\varepsilon_{n+1})\ar[dd] \\
    Q(C(a_1\dots a_n)) \ar[rr] & \ins{\eqref{E:potentielconvertir}} &  Q(\bq) \ar[rr] &  \ins{\eqref{E:CDsanssanscompteur}} &  Q(\br) \\
    & (a_1,1)\dots(a_{n},1) \ar[dd]& & (c_{n+1},\varepsilon_{n+1})\ar[dd]\\
	C_{\blank}^{\$} \ar[rr] & \ins{\eqref{E:deltaC4}} & C_{a_{n}}^{\$} \ar[rr] & & \delta_{(c_{n+1},\varepsilon_{n+1})}(C_{a_{n}}^{\$})\\
	& (a_1,0)\dots(a_{n},0)\ar[dd] & &(c_{n+1},1-\varepsilon_{n+1})\ar[dd] \\
    Q(C(a_1\dots a_n)) \ar[rr] & \ins{\eqref{E:potentielconvertir}}  &  Q(\bq) \ar[rr] & \ins{\eqref{E:CDsanssanscompteur}}  &  Q(\br) \\
    & (a_1,1)\dots(a_{n},1) \ar[dd]&  & (c_{n+1},1-\varepsilon_{n+1})\ar[dd]\\
	C_{\blank}^{\$} \ar[rr] & \ins{\eqref{E:deltaC4}}  & C_{a_{n}}^{\$} \ar[rr] & & \delta_{(c_{n+1},1-\varepsilon_{n+1})}(C_{a_{n}}^{\$})\\
	& (a_1,0)\dots(a_{n},0) & &(c_{n+1},\varepsilon_{n+1})
  }
\]
If $\varepsilon_{n+1}=0$, then $\delta_{(c_{n+1},\varepsilon_{n+1})}(C_{a_{n}}^{\$})=\neg$ and $\delta_{(c_{n+1},1-\varepsilon_{n+1})}(C_{a_{n}}^{\$})=C_{c_{n+1}}^{\$}$. Therefore, we find
\begin{align*}
(Q(\br)  \delta_{(c_{n+1},\varepsilon_{n+1})}&(C_{a_{n}}^{\$}) Q(\br)\delta_{(c_{n+1},1-\varepsilon_{n+1})}(C_{a_{n}}^{\$}))^2\\ &= (Q(\br)\neg Q(\br) C_{c_{n+1}}^{\$})^2\\
&\equiv (\neg C_{c_{n+1}}^{\$})^2 & &\text{by Lemma~\ref{L:inverseQuneg},}\\
&\equiv \id & &\text{by Lemma~\ref{L:inversestate}.}
\end{align*}
Therefore $\sigma_{Q(C(a_1\dots a_n))C_{\blank}^{\$}}^4(\bx)=\bx$. Similarly, if $\varepsilon_{n+1}=1$, then $\delta_{(c_{n+1},\varepsilon_{n+1})}(C_{a_{n}}^{\$})=C_{c_{n+1}}^{\$}$ and $\delta_{(c_{n+1},1-\varepsilon_{n+1})}(C_{a_{n}}^{\$})=\neg$. However, it follows from Lemmas~\ref{L:inversestate} and \ref{L:inverseQuneg} that $Q(\br)C_{b_{n+1}}^{\$}Q(\br)\neg$ is of order $2$, implying $\sigma_{Q(C(a_1\dots a_n))C_{\blank}^{\$}}^4(\bx)=\bx$. Therefore, $(3)$ holds. This concludes the proof of Lemma~\ref{L:ordrepetit}.
\end{proof}

We now show that, if a final state appears in configuration, then the order of the corresponding element is two. This is the halting condition for powers of elements in our group~$G$.

\begin{lemma}\label{L:order2}
Assume that there is $b\in T$ such that $t(b,\blank,\blank)\not\in F$. Let $\ba\in T^*$. Assume that one letter of $\ba$ is in $F$. Then the order of ${Q(C(\ba))C_\blank^\$}$ is two in~$G$.
\end{lemma}

\begin{proof}
Lemma~\ref{L:ordrefindeliste} implies that the order of ${Q(C(\ba))C_\blank^\$}$ is at least two. Consider $\ba=a_1\dots a_n$, and let $i\le n$ be such that $a_i\in F$.

Let~$\bx$ in $\Sigma^*$. Write $x=(c_1,\varepsilon_1)\dots(c_N,\varepsilon_N)$. We can assume that $N\ge n$. Note that either $c_i\not=a_i$ or $c_i\in F$. Hence, it follows from Lemma~\ref{L:ordrepetit}(1) and (2) that $\sigma_{Q(C(\ba))C_\blank^\$}^2(\bx)=\bx$. Therefore the order of ${Q(C(\ba))C_\blank^\$}$ is two.
\end{proof}

We now connect the orders of elements defined by two successive configurations. This corresponds to one step of computation of the automaton.

\begin{lemma}\label{L:ordreunpasinduction}
Assume that there is $b\in T$ such that $t(b,\blank,\blank)\not\in F$. Let $\ba\in (T\setminus F)^*$. If the order of $Q(C(\tau(\ba)))C_\blank^\$$ is finite in~$G$, then the order of $Q(C(\ba))C_\blank^\$$ is at most twice the order of $Q(C(\tau(\ba)))C_\blank^\$$ in~$G$.
\end{lemma}

\begin{proof}
Consider $\ba=a_1\dots a_n\in (T\setminus F)^n$. Denote by $m$ the order of $Q(C(\tau(\ba)))C_\blank^\$$. It follows from Lemma~\ref{L:ordrefindeliste} that $m$ is even.

Let $\bx\in\Sigma^*$ of length $N>n$. Assume that there is $k\le n$ such that $x_k\not=(a_k,0)$. It follows from Lemma~\ref{L:ordrepetit}(2) that $\sigma_{Q(C(\ba))C_\blank^\$}^2(\bx)=\bx$.  Similarly, if $x_{n+1}\not=(\$,0)$ and $x_{n+1}\not=(\$,1)$, then, by Lemma~\ref{L:ordrepetit}(3), we have $\sigma_{Q(C(\ba))C_\blank^\$}^4(\bx)=\bx$. Therefore, in both cases, as $m$ is even, we find $\sigma_{Q(C(\ba))C_\blank^\$}^{2m}(\bx)=\bx$.

We can assume that, for every $ k\le n$, we have $x_k=(a_k,0)$. If $x_{n+1}=(\$,1)$, then we set $\br=\rho_{(\$,1)}(\bq)$. Then, we obtain the following cross diagram
\[
 \xymatrix{
	& (a_1,0)\dots(a_{n},0)\ar[dd] & &(\$,1)\ar[dd] \\
    Q(C(a_1\dots a_n)) \ar[rr] & \ins{\eqref{E:potentielconvertir}} &  Q(\bq) \ar[rr] &  \ins{\eqref{E:CDsanssanscompteur}} &  Q(\br) \\
    & (a_1,1)\dots(a_{n},1) \ar[dd]& &(\$,1)\ar[dd]\\
	C_{\blank}^{\$} \ar[rr] & \ins{\eqref{E:deltaC4}} & C_{a_{n}}^{\$} \ar[rr] & \ins{\eqref{E:deltaC1}} & C_{\blank}^{\$}\\
	& (a_1,0)\dots(a_{n},0) & &(\$,0)
}
\]
Hence $\by=\sigma_{Q(C(\ba))C_\blank^\$}(\bx)$ starts with $(a_1,0)\dots(a_{n},0)(\$,0)$. Moreover, the size of the orbits of $\sigma_{Q(C(\ba))C_\blank^\$}$ on~$\bx$ and on~$\by$ are the same. Thus we can assume $x_{n+1}=(\$,0)$.

Now, Lemma~\ref{L:deltaZa2} implies
\[
\sigma_{Q(C(\ba))C_\blank^\$}^{2m}(\bx) = \sigma_{(Q(C(\ba))C_\blank^\$)^2}^{m}(\bx) = x_1\dots x_n x_{n+1} \sigma_{Q(C(\tau(\ba)))C_\blank^\$}^m(x_{n+2}\dots x_N)=\bx\,.
\]
Therefore, the order of ${Q(C(\ba))C_\blank^\$}$ is at most $2m$.
\end{proof}

Using Lemmas~\ref{L:order2},~\ref{L:ordreunpasinduction}, and~\ref{L:minorerordre}, we inductively deduce a characterization of the order of $Q(C(\ba))C_\blank^\$$ depending on the number of steps required for the cellular automaton to halt when it starts from a configuration $\ba$:

\begin{theorem}\label{T:ordrenotreelement}
Let $\ba\in T^*$, and let $k\ge 1$. Assume that, in $\ba,\tau(\ba),\dots, \tau^k(\ba)$, there is no letter of $F$, and that, in $\tau^{k+1}(\ba)$, there is a letter of $F$. Then the order of $Q(C(\ba))C_\blank^\$$ is $2^{k+2}$ in~$G$.
\end{theorem}

We are now ready to prove that the element of the group~$G$ associated with a configuration~$\ba$ is of finite order if, and only if, the cellular automaton~$(T,t)$ halts when it starts from $\ba$:

\begin{corollary}\label{C:ordrefiniarret}
For every $\ba=a_1\dots a_n$ in~$(T\setminus F)^*$, the following are equivalent: 
\begin{enumerate}
\item There is $k$ such that  $\sigma_{Q(C_{\blank}^{a_1}\dots C_{\blank}^{a_n})C_\blank^\$}$ is of order $2^{k+2}$ in~$G$.
\item $\sigma_{Q(C_{\blank}^{a_1}\dots C_{\blank}^{a_n})C_\blank^\$}$ is of finite order in~$G$.
\item There is $k$ such that $\tau^{k+1}(\ba)$ contains a symbol  of~$F$.
\item There is $k$ such that $[\dots [[\sigma_{Q(C_{\blank}^{a_1}\dots C_{\blank}^{a_n})C_\blank^\$},\sigma_{\neg C_\blank^\$}],\sigma_{\neg C_\blank^\$}],\dots,\sigma_{\neg C_\blank^\$}]$, $2^{k+2}$ commutators, is the identity.
\end{enumerate}
\end{corollary}

\begin{proof}
As the word $Q(C_{\blank}^{a_1}\dots C_{\blank}^{a_n})$ and the one-letter word $C_\blank^\$$ are palindromes, it follows from Corollary~\ref{C:conjuguepuissance} that $(1)$ and $(4)$ are equivalent. The implication $(1)\Longrightarrow(2)$ is immediate.

Assume that $(3)$ fails, and let $k\ge 1$. Then, no symbol of $F$ appears in the words $\ba, \tau(\ba),\dots, \tau^k(\ba)$. It follows from Lemma~\ref{L:minorerordre} that the order of ${Q(C_{\blank}^{a_1}\dots C_{\blank}^{a_n})C_\blank^\$}$ is at least $2^{k+2}$. As this is true for each integer $k$, the order is infinite. So the implication $(2)\Longrightarrow (3)$ holds.

Finally, the implication $(3)\Longrightarrow (1)$ follows from Theorem~\ref{T:ordrenotreelement}.
\end{proof}

\section{Simulation of a Turing Machine}\label{S:indecidabilite}

It is known that one can simulate a Turing machine using a cellular automaton. There exists a very small example of a Turing-universal cellular automaton (rule 110, see Cook \cite{Cook1,Cook2}), and even one of a universal cellular automaton with only four states (see Ollinger and Richard \cite{OR}). However, a periodic tape is required, and, moreover, there is no halting state. Thus, we shall resort here to a more direct simulation of a universal Turing machine, based on the universal Turing machine constructed by Minsky in \cite{Minsky}. Rogozhin in \cite{R} and Neary and Woods in \cite{NW} gave examples of  smaller universal Turing machines, but the latter do not seem to provide a smaller cellular automaton.

Smith proves in \cite{Smith} that every $m$-state $n$-symbol Turing machine can be simulated by a one-dimensional cellular automaton with $n+2m$ states and neighborhood $\set{-1,0,1}$. The states of the cellular automaton are the symbols of the Turing machine plus two copies of the state set of the Turing machine (in order to indicate the direction the head moves). However, as mentioned by Lindgren and Nordahl in \cite{LN}, we only need to duplicate those states such that the head can, for an appropriate symbol, go to the left and to the right. In this way, starting from Minsky's universal Turing machine $M$ with 7 states and 4 symbols (cf. \cite{Minsky}), one obtains a Turing-universal cellular automaton with $13$ states. Moreover, we can add a symbol~$f$, and complete the cellular automaton writing~$f$ whenever the transition is undefined (i.e. $M$ halts).

This gives us a complete Turing-universal cellular automaton $(T,t)$ with $\card T=14$. Denote by $\blank$ the blank symbol of $M$ (observe that $\blank\in T$). Set $F=\set{f}$. Minsky's Turing machine $M$ only needs a one-way infinite tape to simulate Turing machines (through the simulation of tag systems). Therefore, the construction given in Notation~\ref{N:applyingcellularautomata} simulates one step of $M$ when starting from a (finite) configuration corresponding to the simulation of a tag system.

Considering the Mealy automaton constructed in Section~\ref{S:construction}, the main result now directly follows from Corollary~\ref{C:ordrefiniarret}, and from the fact that the halting problem is undecidable (Turing \cite{T}).

\begin{theorem}
There is an automaton group $G$ such that both the order problem and the Engel problem are undecidable in $G$.
\end{theorem}

\begin{remark}
As $\card T=14$, it follows from the definition of $\Sigma$ and $A$ that the number of symbols is $\card\Sigma=30$, and the number of states is $\card A= 421$.
\end{remark}

We conclude the paper with an alternative construction that allows one to reduce the number of states and symbols. Fix a Turing machine $M=(S,Q,s_0,q_0,\varphi)$, where $S$ is the state set, $Q$ is the alphabet, $s_0\in S$ is the initial state, $q_0\in Q$ is the blank symbol, and $\varphi\colon E\to S\times Q\times \set{\leftarrow,\rightarrow}$ is the transition map, where $E\subseteq S\times Q$. The Turing machine halts if it reaches a configuration such that it cannot continue. That is, the state is $s\in S$ and the head read a symbol $q\in Q$, such that $(q,s)\not\in E$. Assume in addition that the Turing machine operates on a tape indexed by $\N$.

Set $T=S\sqcup Q$. A configuration of the Turing machine $M$ can be represented by a word $q_1\dots q_k s q_{k+1}\dots q_n$ when the head points to the cell number $k$ in state $s$.

A cellular automaton with neighborhood $\set{-1,0,1}$ is not sufficient to simulate the Turing machine directly, but the Neighborhood $\set{-2,-1,0,1}$ is. However, in order to optimize the size, we directly give the construction of a Mealy automaton, without going through the description of the cellular automaton.

As before we set $T'=T\cup\set{\$}$ and $\Sigma=T'\times \set{0,1}$. We do not need to memorize much in order to compute one symbol of the next configuration. The set of states is defined in the following way
\[
A=\set{\neg,C^{\$},W}\cup \set{C}\times Q\times S \cup \set{C}\times Q\times Q \cup \set{C}\times S\times Q \cup \set{P}\times Q \cup \set{P}\times S \cup \set{W}\times Q\,.
\]

Set $\blank=q_0$. We do not need final states as before, because the latter correspond to subwords of the form $as$ where $a\in Q$, $s\in S$, and $(s,a)\not\in E$. The \emph{modifying} states are the states in $\set{\neg,C^{\$}}\cup \set{C}\times Q\times S \cup \set{C}\times Q\times Q \cup \set{C}\times S\times Q$. As before, for all $p\in A$ and $x\in\Sigma$, we set $\sigma(p,x)=x^\bot$ if, and only if, $p$ is a modifying state, otherwise we set $\sigma(p,x)=x$.

The transition map is defined below. For all $a,b,d\in Q$, all $r,s\in S$, and all $u,v,w\in Q\cup S$, we consider $\delta\colon A\times\Sigma \to A$ so that $\delta(p, x)$ is:

\newcommand\BOX[1]{\leavevmode\hbox to 12mm{$\bullet\ #1$\hfil}}

\BOX{\neg}if $p=\neg$,

\BOX{C^\$}else if $p=C^\$$ and $x=(\$,1)$,

\BOX{C_{\blank}^{\blank}}else if $p=C^\$$ and $x=(\$,0)$,

\BOX{C_\blank^u}else if $p=P_u$ and $x=(\$,\varepsilon)$,

\BOX{C_{\blank}^{\blank}}else if $p=W_a$ and $x=(\$,\varepsilon)$,

\BOX{\neg}else if $p=W$ and $x=(\$,\varepsilon)$,

\BOX{\neg}else if $p=C_u^v$ and $x=(\$,\varepsilon)$,

\BOX{P_u}else if $p=P_u$ and $x=(v,\varepsilon)$,

\BOX{\neg}else if $p=C_u^v$ and $x=(w,0)$, with $w\not=v$,

\BOX{C_d^b}else if $p=C_u^b$ and $x=(d,1)$,

\BOX{C_r^b}else if $p=C_a^b$ and $x=(s,1)$, $(s,a)\in E$, and $\varphi(s,a)=(d,r,\rightarrow)$,

\BOX{C_a^b}else if $p=C_a^b$ and $x=(s,1)$, $(s,a)\in E$, and $\varphi(s,a)=(d,r,\leftarrow)$,

\BOX{\neg}else if $p=C_a^b$ and $x=(s,1)$, and $(s,a)\not\in E$, 

\BOX{P_s}else if $p=C_s^b$ and $x=(b,0)$,

\BOX{W_b}else if $p=C_a^b$ and $x=(b,0)$,

\BOX{\neg}else if $p=C_a^s$ and $x=(r,1)$,

\BOX{C_b^s	}else if $p=C_a^s$ and $x=(b,1)$,

\BOX{\neg}else if $p=C_a^s$ and $x=(s,0)$ and $(s,a)\not\in E$,

\BOX{P_d}else if $p=C_a^s$ and $x=(s,0)$, $(s,a)\in E$ and $\varphi(s,a)=(d,r,\leftarrow)$,

\BOX{W}else if $p=C_a^s$ and $x=(s,0)$, $(s,a)\in E$ and $\varphi(s,a)=(d,r,\rightarrow)$,

\BOX{P_a}else if $p=W_a$ and $x=(b,\varepsilon)$

\BOX{P_r}else if $p=W_a$ and $x=(s,\varepsilon)$, $(s,a)\in E$, and $\varphi(s,a)=(d,r,\leftarrow)$,

\BOX{P_a}else if $p=W_a$ and $x=(s,\varepsilon)$, $(s,a)\in E$, and $\varphi(s,a)=(d,r,\rightarrow)$,

\BOX{\neg}else if $p=W_a$ and $x=(s,\varepsilon)$, $(s,a)\not\in E$

\BOX{P_a}else if $p=W$ and $x=(a,\varepsilon)$,

\BOX{\neg}else if $p=W$ and $x=(s,\varepsilon)$.

Then, the above proofs can be adapted to obtain the following result:

\begin{theorem}
Let  $\bq=q_1\dots q_k s q_{k+1}\dots q_n$ be a configuration of $M$. Assume that, starting from this configuration, $M$ only visits a one-way infinite tape. Then $C(\blank \bq\blank) C^{\$}$ is of finite order if, and only if, $M$ halts when starting from the configuration~$\bq$.
\end{theorem}

In particular, if $M$ is a universal Turing machine requiring only a one-way infinite tape, the corresponding automaton group has an undecidable order problem. When considering the universal Turing machine with six states and four symbols constructed by Neary and Woods in \cite{NW}, we obtain a Mealy automaton with $22$ symbols and $81$ states.
However, the states $C_r^b$ with $r\in S$ and $b\in Q$ can only appear if we have $(s,a)\in E$, and $d\in Q$ such that $\varphi(s,a)=(d,r,\rightarrow)$. By looking at the transition table of the Neary and Woods Turing machine, we can see that one of the states is missing, so that we can remove the corresponding states from $A$, thus obtaining a Mealy automaton with $77$ states.

Using similar ideas, and starting from the universal Turing machine with $9$ states and $3$~symbols of Neary and Woods, we can see on the transition table that three states cannot appear when the Turing Machine move to the right. Hence we obtain a Mealy automaton with $26$ symbols and $72$ states.

\section{Funding}
The author was supported by CONICYT, Proyectos Regulares FONDECYT no. 1150595, and by project number P27600 of the Austrian Science Fund (FWF).

\end{document}